\documentclass[11pt]{article}
\usepackage[utf8]{inputenc}
\usepackage{lmodern}
\usepackage{subfiles}
\usepackage{enumitem}
	\setenumerate{ref=(\roman*), label={\normalfont(\roman*)}, itemsep=0em} 
        
\usepackage{amsfonts}
\usepackage{amsthm}
\usepackage{amsmath}
\usepackage{amssymb}
\usepackage{amscd}
\usepackage{mathrsfs}
\usepackage{mathtools}
\usepackage{bm}
\usepackage{esint}
\usepackage{crossreftools}

\usepackage[margin=3cm]{geometry}
\usepackage{setspace}
\usepackage{indentfirst}
\usepackage{graphicx}
\usepackage{graphics}
\usepackage{pgf,tikz}
\usepackage{tikz-cd}
\usepackage{color}
\usepackage{pict2e}
\usepackage{epic}
\usepackage{epstopdf}
\usepackage{titlesec, titlefoot}
	\titleformat{\section}[block]{\Large\bfseries\filcenter}{\thesection}{1em}{}
\usepackage{commath}
\usepackage{float}
\usepackage{caption}
\usepackage{etoolbox}
\usepackage[affil-it]{authblk}
\usepackage{combelow}

\usepackage[hidelinks, bookmarksdepth=3]{hyperref}
\hypersetup{bookmarksopen=true} 
\usepackage{hypcap}

\usepackage{graphicx}
\usepackage{subcaption}

\usepackage{algorithm}
\usepackage{algorithmic}

\graphicspath{{./Pictures/}}
\allowdisplaybreaks

\usepackage{listings}
\usepackage{fancybox}

\theoremstyle{plain}

\renewcommand*\thesection{\arabic{section}}
\numberwithin{equation}{section} 

\theoremstyle{plain}
\newtheorem{thm}{Theorem}
\newtheorem{theorem}{Theorem}
\newtheorem{lemma}[thm]{Lemma}
\newtheorem{prop}[thm]{Proposition}

\numberwithin{thm}{section} 

\theoremstyle{definition}

\newtheorem{remark}[thm]{Remark}

\newcommand{\thistheoremname}{}
\newtheorem{genericthm}[thm]{\thistheoremname}

\newcommand{\thistheoremnames}{}
\newtheorem*{genericthms}{\thistheoremnames}
\newenvironment{para*}[1]
  {\renewcommand{\thistheoremnames}{#1}%
   \begin{genericthms}}
  {\end{genericthms}}

\expandafter\let\expandafter\oldproof\csname\string\proof\endcsname
\let\oldendproof\endproof
\renewenvironment{proof}[1][\proofname]{%
  \oldproof[\upshape \bfseries #1]%
}{\oldendproof}

\makeatletter
\def\@makechapterhead#1{%
  \vspace*{50\p@}%
  {\parindent \z@ \raggedright \normalfont
    \interlinepenalty\@M
    \Huge\bfseries  \thechapter.\quad #1\par\nobreak
    \vskip 40\p@
  }}
\makeatother

\usepackage{tocloft}
\cftsetindents{section}{0em}{2em}
\cftsetindents{subsection}{0em}{2em}
\def \R {\mathbb{R}}

\def \N{\mathbb{N}}
\def \D{\textup{D}}

\def \e{\varepsilon}

\def \p{\partial}

\def \tp{\textup}

\def \d{\,\mathrm{d}}
\def \Div{\tp{div}\,}

\title{Regularity for quasilinear vectorial elliptic systems through an iterative scheme with numerical applications}
\author{Lukas Koch}

\affil[1]{\small University of Oxford, Andrew Wiles Building Woodstock Rd, Oxford OX2 6GG, United Kingdom 
\protect \\
  {\tt{kochl@maths.ox.ac.uk}}
  \vspace{1em} \ }

\newcommand\blfootnote[1]{%
  \begingroup
  \renewcommand\thefootnote{}\footnote{#1}%
  \addtocounter{footnote}{-1}%
  \endgroup
}

\begin{document}

\maketitle
\blfootnote{
\emph{2010 Mathematics Subject Classification:} 35J57, 65N30\\ 
\emph{Keywords.} Regularity, iterative Galerkin schemes, p-growth\\
\emph{Acknowledgements.} L.K. was supported by the Engineering and Physical Sciences Research Council [EP/L015811/1].
}

\begin{abstract}
We consider an iterative procedure to solve quasilinear elliptic systems with $p$-growth. The scheme was first considered by \textsc{Koshelev} in the quadratic case $p=2$. We present numerical applications as well as applications to higher regularity properties.
\end{abstract}

\section{Introduction}
Let $1<p<\infty$ and consider an open and bounded domain $\Omega\subset \R^n$.
This paper is concerned with problems of the form: Seek $u\in W^{1,p}_g(\Omega)=g+W^{1,p}_0(\Omega,\R^N)$ such that
\begin{align}\label{eq:KoshProblem}
\Div a(x,\D u) =& \,\Div f \quad \text{ in } \Omega
\end{align}
where ${a\colon \Omega\times \R^{N \times n}\to \R^{N \times n}}$ is a matrix-valued elliptic structure field satisfying controlled $p$-growth conditions, $f\in L^{p'}(\Omega,\R^{N\times n})$ and $g\in W^{1,p}(\Omega,\R^N)$. Here $p'$ denotes the H\"older conjugate of $p$ and \eqref{eq:KoshProblem} is to be understood in the sense of distributions.
To be precise, we make the following assumptions on $a$:
\begin{enumerate}
\item[{\crtcrossreflabel{(A1)}[ass:A1]}] For every $z\in \R^{N\times n}$, $a(\cdot,z)$ is measurable in $\Omega$ and, for  almost every $x\in\Omega$, $a(x,\cdot)$ is continuously differentiable in $\R^{N\times n}$.
\item[{\crtcrossreflabel{(A2)}[ass:A2]}] There are $\Lambda_a\geq\lambda_a>0$ and $\mu\geq 0$ such that, for almost every $x\in\Omega$ and for all $z,\xi\in \R^{N \times n}$, $$\lambda_a (\mu^2+|z|^2)^\frac{p-2}{2}|\xi|^2\leq \p_z a(x,z)\xi\cdot \xi \leq \Lambda_a (\mu^2+|z|^2)^\frac{p-2}{2}|\xi|^2.$$
\item[{\crtcrossreflabel{(A3)}[ass:A3]}] There is $c>0$, such that, for all $x\in\Omega$, $z\in \R^{N \times n}$, $|a(x,z)|\leq  c\left(1 +|z|^{p-1}\right)$.
\end{enumerate}
For simplicity of presentation, in addition to $\mathrm{\ref{ass:A1}-\ref{ass:A3}}$, we make the further assumption that 
\begin{align}\tag{A4}\label{ass:A4}
\text{For almost every } x\in\Omega \text{ and every } z\in \R^{N\times n},\, \p_z a(x,z) \text{ is symmetric.}
\end{align}
\begin{remark}
We have chosen to focus on systems of the form $\Div a(x,\D u)=\Div f$, however with minor changes to $\mathrm{\ref{ass:A1}-\ref{ass:A3}}$ all the theory we develop applies equally to systems of the form $\Div a(x,u,\D u)=\Div f$ where $a\colon \Omega\times \R^N\times \R^{N\times n}\to\R^{N\times n}$. 
Further \eqref{ass:A4} can be dropped if the results of this paper are modified as indicated in Remark \ref{rem:nonsymmetric}.
\end{remark}

We will only concern ourselves with the vectorial case $N\geq 2$ as the theory in the scalar case $N=1$ is different and much more can be said.

The study of elliptic systems of the form \eqref{eq:KoshProblem} is well established with a very long list of important results. For an introduction we refer to \cite{Giaquinta1983} and \cite{Giusti2003} as well as the classical \cite{Ladyzhenskaya1968}.
It is well known that in general, without stronger structural assumptions than ellipticity and controlled $p$-growth, solutions to \eqref{eq:KoshProblem} will not be $C^{1,\alpha}$- or even $C^{0,\alpha}$- regular. In fact, problems already arise in the simplest case when $p=2$. An initial counterexample for systems with $a\equiv a(x,z)$ (with $n\geq 3$) was developed in \cite{DeGiorgi1968}, where a quadratic system was given for which solutions are not H\"older-continuous and, in fact, fail to be bounded. While the $x$-dependence of the system in \cite{DeGiorgi1968} is discontinuous, almost at the same time an example was found of a higher-order elliptic operator with analytic coefficients but discontinuous solutions \cite{Mazya1968}. We remark at this point that, if $n=2$, solutions to \eqref{eq:KoshProblem} are H\"older continuous \cite{Morrey1938}. When $a\equiv a(x,y,z)$ does indeed depend on $y$, it is possible to give examples where $a$ is analytic, but solutions to \eqref{eq:KoshProblem} are nowhere continuous, see \cite{Soucek1984,John1989}, building on examples with analytic fields $a$ but discontinuous solutions in \cite{Miranda1968}. In this set-up, an example of a system in optimal dimensions $n>2$, $N=2$ with discontinuous solutions is given in \cite{Frehse1973}. However, the lack of regularity occurs already without $x$-dependence of the coefficients. Even when $a\equiv a(z)$ is quadratic and analytic, $C^{1,\alpha}$-regularity of solutions need not hold, if $n>2$, as \cite{Necas1980} shows.

Due to this lack of regularity, the general regularity theory of systems of type \eqref{eq:KoshProblem} proceeds through notions of partial regularity, that is, regularity outside of a context-dependent small relatively closed set. Since our focus here lies on results holding everywhere in the domain, we refer the reader to \cite{Giaquinta1983,Giusti2003,Mingione2006} for results and references in this direction.

Regularity results holding in the full domain are known only in a few special cases. The radial Uhlenbeck structure $a\equiv a(\left|z \right|)z$, originally developed in \cite{uhlenbeck1977}, is well known to guarantee full $C^{1,\alpha}$-regularity. In \cite{Tolksdorf1983}, the result is generalised to fields of the form $g(z,z)z$, where $g(\cdot,\cdot)$ is a symmetric, positive definite bilinear form. In particular, this theory covers the model case of the $p$-Laplace operator $a(x,y,z)=|z|^{p-2}z$.
There now exists a vast literature concerning the regularity theory of \eqref{eq:KoshProblem} if $a\equiv a(|z|)z$ and we refer the reader to \cite{Giaquinta1986,Marcellini1996,Marcellini2006,Lieberman2006,Diening2009,Diening2019} for further results and references.

A second direction of everywhere regularity results, in the case where $a\equiv a(z)$, concerns the case where the modulus of continuity of $\p_z a(z)$ is not too large compared to the ellipticity constant of $a(z)$. $C^{1,\alpha}_{\tp{loc}}$-regularity of solutions to \eqref{eq:KoshProblem} is shown in this set-up in \cite{Danecek2004,Danecek2007,Danecek2013}. 

A further direction, closely related to the results of this paper, are results of Cordes-Nirenberg type. For linear elliptic PDE, the classical Cordes-Nirenberg results \cite{Nirenberg1954,Cordes1956} state that solutions of
\begin{align}\label{eq:CordesNirenberg}
 a_{ij}(x)\partial_{ij}u(x)=f(x)
\end{align}
have interior $C^{1,\alpha}$ regularity, if $f\in L^\infty$ and $|a_{ij}-\delta_{ij}|<\varepsilon$ for sufficiently small $\varepsilon$. That is, $C^{1,\alpha}$-regularity is obtained when the coefficients $\{a_{ij}\}$ are sufficiently close to the identity matrix in an appropriate sense. Note that the identity matrix is nothing but the coefficients corresponding to the Laplacian. In \cite{Caffarelli1989} the result is extended to fully non-linear elliptic equations with solutions in the sense of viscosity solutions. A result of a similar spirit is \cite{Kristensen2013}, where local $BMO$-regularity for extremals of (a suitable relaxed version of) $\tp{dist}_K^2(\D u)$ is shown. Here $\tp{dist}_K(z)$ denotes the distance of $z\in \R^{N\times n}$ to the compact set $K\subset\R^{N\times n}$.

We also comment on the case where $a(x,z)=a(x)z+b(x,z)z$ and $\|b(x,z)\|\leq c(n)\lambda_a$. Here $a(x)$ is a positive definite matrix, with smallest eigenvalue uniformly bounded below in $\Omega$ by $\lambda_a>0$ and $c(n)$ is a constant depending only on $n$. If $c(n)$ is sufficiently small, then bounded solutions of \eqref{eq:KoshProblem} are $\alpha$-H\"older-continuous \cite{Wiegner1982}.

In this paper we obtain regularity results for fields $a$ that are close to a suitable reference field $b$, see the start of Section \ref{sec:KoshelevBasic} for a precise definition of our notion of reference field. Suppose $a,b$ satisfy $\mathrm{\ref{ass:A1}}$ and $\mathrm{\ref{ass:A2}}$, and let $A(x,z)=\p_z a(x,z), B(x,z)=\p_z b(x,z)$. Denote by $\lambda_{B^{-1}A}(x,z), \Lambda_{B^{-1}A}(x,z)$ the smallest and largest eigenvalue of $B^{-1}(x,z)A(x,z)$, respectively. 
We note that then $\lambda_{B^{-1}A}$ has a lower bound, and $\Lambda_{B^{-1}A}(x,z)$ has an upper bound, that is independent of $x,z$. Denote these bounds by $\lambda_{a,b}$ and $\Lambda_{a,b}$, respectively. We set
\begin{align}\label{def:K}
K_{a,b}=\frac{\Lambda_{a,b}-\lambda_{a,b}}{\lambda_{a,b}+\Lambda_{a,b}}
\end{align}
We will operate in a regime where $K_{a,b}$ is sufficiently small and remark that in all cases we are able to make the smallness assumption on $K_{a,b}$ explicit. In this regime we are able to transfer existence and regularity results for solutions of the system $\Div b(x,\D u)=\Div f$ to solutions of the system $\Div a(x,\D u)=\Div f$.

The observation of \cite{Koshelev1991} is that, if $a$ is a matrix-valued elliptic structure field satisfying controlled quadratic growth, then considering the scheme
\begin{align*}
\begin{cases}
-\Delta u_{n+1}=-\Delta u_n - \gamma \Div a(x,\D u_n) +\gamma \Div f &\quad \text{ in } \Omega\\
u_{n+1}= 0 &\quad \text{ on } \p\Omega,
\end{cases}
\end{align*}
with any choice of $u_0\in W^{1,2}_g(\Omega)$,
gives a sequence $\{u_n\}$ that converges in $W^{1,2}(\Omega)$ to a solution of \eqref{eq:KoshProblem}. The problem is to be understood in the sense of distributions. \cite{Koshelev1991} further shows convergence of the scheme in appropriate Morrey spaces, obtaining $C^{0,\alpha}$- and $C^{1,\alpha}$-regularity of solutions to \eqref{eq:KoshProblem} under (sharp in the case of $C^{0,\alpha}$) smallness assumptions on $K$. We will return to this point shortly.

We extend this result to matrix-valued elliptic structure fields with $p$-growth.
\begin{theorem}\label{thm:KoshelevBasicIntro}
Let $1<p<\infty$, $f\in L^{p'}(\Omega,\R^{N \times n})$ and $g\in W^{1,p}(\Omega,\R^N)$.
Suppose $a, b$ satisfy $\mathrm{\ref{ass:A1}-\ref{ass:A3}}$ and \eqref{ass:A4}. Assume that the problem
$$
\Div b(x,\D u)= \Div F \text{ in } \Omega
$$
has a unique solution in $W^{1,p}_g(\Omega)$ for any choice of $F\in L^{p'}(\Omega)$. 
If
\begin{align*}
\begin{cases}
\dfrac{\Lambda_b}{\lambda_b}K_{a,b} \left(\dfrac{3+\sqrt{5}}{2}\right)^\frac{p-2}{2p} 6^{p-2}(p-1)< 1 &\quad\text{ for } p\geq 2 \\[20pt]
\dfrac{\Lambda_b}{\lambda_b}K_{a,b} \left(\dfrac{3-\sqrt{5}}{2}\right)^\frac{(p-2)}{4}\dfrac{2^{2-p}}{p-1}< 1 &\quad\text{ for } p\leq 2,
\end{cases}
\end{align*}
then there is $\gamma>0$ such that, if 
${u_0\in W^{1,p}_g(\Omega)}$, and ${u_{n+1}\in W^{1,p}_g(\Omega)}$ is inductively defined to be the weak solution of
\begin{align}\label{eq:KoshIteration}
\Div b(x,\D u_{n+1}) = \Div( b(x,\D u_n)-\gamma a(x,\D u_n)+ \gamma f) \text{ in } \Omega
\end{align}
in the class $W^{1,p}_g(\Omega)$, then $u_n\to u$ in $W^{1,p}(\Omega)$. Here $u$ is the unique solution in $W^{1,p}_g(\Omega)$ of \eqref{eq:KoshProblem}.
\end{theorem}

\begin{remark}
We note that if $p=2$ and $b(x,z) = z$, we recover precisely the situation in \cite{Koshelev1991}.
\end{remark}

The iterative scheme \eqref{eq:KoshIteration} naturally lends itself as a numerical scheme to solve \eqref{eq:KoshProblem} using finite element approximations. Take $g=0$ and suppose $\Omega$ is a regular polytope. Consider a family of shape-regular triangulations $\{T_h\}_{h\in(0,1]}$ of $\Omega$ of mesh-size $h$. Denote by $X_h$ the space of continuous piecewise linear (subordinate to $T_h$) functions. Take $u_0\in X_h$. Define $u_{n+1}^h\in X_h$ inductively as a solution of the problem:
\begin{align}\label{eq:numericalScheme}
\begin{cases}\int_\Omega b(x,\D u_{n+1}^h)\cdot \D \phi\d x = \int_\Omega (b(x,\D u_n^h)-\gamma a(x,\D u_n^h)+\gamma f)\cdot \D \phi\d x \quad &\forall \phi\in X_h\\
 u_{n+1}^h = 0 &\text{ on } \p \Omega
 \end{cases}
\end{align}
Note that due to $\mathrm{\ref{ass:A1}-\ref{ass:A3}}$ and the theory of monotone operators \eqref{eq:numericalScheme} is well-defined. We prove in Section \ref{sec:numerical} that under the assumptions of Theorem \ref{thm:KoshelevBasicIntro} this scheme converges in $W^{1,p}(\Omega)$ to a solution of the problem 
\begin{align*}
\begin{cases}\int_\Omega a(x,\D u_h)\cdot \D \phi = \int_\Omega f\cdot \D \phi\d x \quad &\forall \phi\in X_h\\
u_h = 0 &\text{ on } \p\Omega.
\end{cases}
\end{align*}
Under the additional assumption that,
\textit{if $F\in L^{p'}(\Omega)$, then there is $\alpha>0$ such that the equation $\Div b(x,\D u)= \Div F$ has a solution $u$ in $W^{1+\alpha,p}(\Omega)$, and, moreover, we have the following estimate for some $c>0$:}
\begin{align}\label{eq:higherRegIntro}
\|u\|_{W^{1+\alpha,p}(\Omega)}\leq c\left(1+\|u\|_{W^{1,p}(\Omega)}+\|F\|_{L^{p'}(\Omega)}^{1/(p-1)}\right),
\end{align}
we show that $u_h\to u$ in $W^{1,p}(\Omega)$ as $h\to 0$ where $u\in W^{1,p}_g(\Omega)$ is the solution of \eqref{eq:KoshProblem}. We note that \eqref{eq:higherRegIntro} is satisfied by many fields of interest, e.g. the $p$-Laplacian $b(x,z)=|z|^{p-2}z$, \cite{Simon1981}.

If $p=2$ and $b(x,z)=z$, the proposed numerical scheme falls into the class of iterative linearised Galerkin schemes studied in \cite{Heid2020}. We refer to this reference for an overview and further references regarding such schemes. An advantage of the schemes studied in this paper, and of unified iterative schemes of the form studied in \cite{Heid2020} in general, is that it is possible to show global convergence results at linear rate for a wide class of problems. This should be compared with nonlinear Newton schemes for which convergence can usually only be expected locally, albeit at quadratic rate. An example of a scheme, fitting into the framework of unified iterative schemes, that has been used to study some problems of $p$-Laplace type where $a(x,z)=a(x,|z|)z$ and $a(x,|z|)$ has $p-2$-growth, is the so-called Kačanov iteration, originally introduced in \cite{Kacanov1959}, and defined by solving
\begin{align*}
\Div a(x,|\D u_n|)\D u_{n+1} = \Div f.
\end{align*}
Versions of this scheme have been studied in \cite{Zeidler1988,Han1997,Garau2011,Diening2020a,Heid2021}, but in general have been restricted to the case $p\in[1,2]$. Note that the scheme studied in this paper cannot be used to replace these schemes as it requires the ability to solve \eqref{eq:KoshProblem} numerically for a reference field. However, given a problem to which a known numerical scheme, such as the Ka\v{c}anov iteration can be applied, it enables to solve perturbations of the problem, to which the known scheme cannot be applied directly. For example, in the case of the Ka\v{c}anov scheme, perturbations need not satisfy any structural assumption of the form $a(x,z)=a(x,|z|)z$.

In \cite{Koshelev1991} the iterative process is used to derive sharp conditions on the dispersion $\Lambda_a/\lambda_a$ of elliptic systems with quadratic growth that guarantee $C^{0,\alpha}$-regularity of solutions. An alternative proof was suggested in \cite{Leonardi1999}. Further, under additional assumptions, conditions for $C^{1,\alpha}$-regularity of solutions are presented. In this spirit, we obtain Cordes-Nirenberg type results with regards to Calder\'{o}n-Zygmund estimates for perturbations of fields of $p$-Laplace type.
We recall a result of \cite{Kinnunen2001}:
\begin{theorem}[Theorem 1.6 in \cite{Kinnunen2001}]\label{thm:Kinnunen}
Let $1<p<\infty$ and $\Omega\subset\R^n$ be a bounded domain with $C^{1,\alpha}$-boundary for some $\alpha\in(0,1]$. Suppose that the coefficients of $B$ belong to $VMO(\Omega)$ and that $F\in L^\frac{q}{p-1}(\Omega)$ for some $q>p$.  Then there exists a unique weak solution $u\in W^{1,p}_0(\Omega)$ of $\Div b(x,\D u)=\Div F$. Moreover, there is a constant $C_0>0$ so that $$\int_\Omega |\D u|^q \d x\leq C_0\int_\Omega |F|^\frac{q}{p-1}\d x.$$ Here $C_0$ depends only on $n,p,q,\lambda,\Lambda$, the VMO data of $B$ and $\p\Omega$. In particular, this implies that $u\in W^{1,q}_0(\Omega)$.
\end{theorem}

We obtain the following extension:
\begin{theorem}
Let $1<p<q<\infty$. Suppose $\Omega$ is a $C^{1,\alpha}$ domain for some $0<\alpha\leq 1$. Take $f\in L^{q/(p-1)}(\Omega)$ and let $g=0$. Assume $b$ and $C_0$ are as in Theorem \ref{thm:Kinnunen}.
  Assume that $a$ satisfies assumptions $\mathrm{\ref{ass:A1}-\ref{ass:A3}}$ and \eqref{ass:A4} with $\mu=0$ and that the assumptions of Theorem \ref{thm:KoshelevBasicIntro} are satisfied. If
$$\dfrac{C_0^{(p-1)/q}\Lambda_b K_{a,b}}{(p-1)}<1,$$ then the solution $u$ of \eqref{eq:KoshProblem} satisfies
\begin{align}\label{eq:IntroHigherInteg}
\|\D u\|_{L^q(\Omega)}\lesssim \left(1+\|f\|_{L^{q/(p-1)}}^\frac{1}{p-1}\right).
\end{align}
\end{theorem}
Further we obtain a similar result with regards to weighted estimates for perturbations of fields of $p$-Laplace type. We recall the following special case of the main result of \cite{Phuc2011}:
\begin{theorem}[c.f. Theorem 2.1 in \cite{Phuc2011}]\label{thm:higherRegMengesha} Let $1<p<q<\infty$ and let $w$ be an $A_{q/p}$ weight. Suppose $\Omega$ is a $C^1$-domain. Suppose $B\in VMO(\Omega)$ and $B$ satisfies 
\begin{align*}
\lambda |\xi|^2\leq B(x)\xi\cdot \xi\leq \Lambda |\xi|^2.
\end{align*} 
Then there exists a positive constant $C_2>0$ such that the following holds. For a given vector field ${F\in L^{q(p-1)}_w(\R^n)}$, there is a unique weak solution ${u\in W^{1,p}_0(\Omega)}$ of $\Div b(x,\D u)=\Div F$, satisfying $\D u\in L^q_w(\Omega)$ with the estimate
\begin{align*}
\|\D u\|_{L^q_w(\Omega)}\leq C_2 \|f\|_{L^\frac{q}{p-1}_w(\Omega)}.
\end{align*}
Here the constant $C_2$ depends only on $n, p, q, R, \lambda,\Lambda$ and $[w]_{q/ p}$.
\end{theorem}

We use the Koshelev iteration to obtain the following extension.
\begin{theorem}
Let $1<p<q<\infty$ and $w\in A_{q/p}$. Assume $\Omega$ is a $C^1$ domain and let $g=0$, $f\in L^\frac{q}{p-1}_w(\Omega)$. Suppose $b$ satisfies the assumptions of Theorem \ref{thm:higherRegMengesha}. Let $C_2$ be the constant from Theorem \ref{thm:higherRegMengesha}.
Suppose $a,b$ satisfy the assumptions of Theorem \ref{thm:KoshelevBasicIntro}.
If $$\frac{C_2^{(p-1)/q} K_{a,b}\Lambda_b}{(p-1)}<1,$$ then the solution $v$ of \eqref{eq:KoshProblem} satisfies the estimate
\begin{align}\label{eq:IntroWeighted}
\|\D v\|_{L^q_w(\Omega)}\leq c(C,p,q)\left(1+\|f\|_{L^\frac{q}{p-1}_w(\Omega)}^\frac{1}{p-1}\right)
\end{align}
\end{theorem}
\begin{remark}
The result can easily be extended to Reifenberg-flat domains, where the corresponding estimate for $b$, needed to replace Theorem \ref{thm:higherRegMengesha} is due to \cite{Mengesha2012}.
\end{remark}

The structure of this paper is as follows. In Section \ref{sec:notation} we explain our notation and present a number of preliminary results. In Section \ref{sec:KoshelevBasic} we show convergence in $W^{1,p}(\Omega)$ of the iterative process. In Section \ref{sec:numerical} we present our numerical experiments regarding the iterative process. In Section \ref{sec:higherReg} we use the iterative process to prove higher differentiability and weighted estimates of perturbations to equations with $p$-Laplace structure.

\section{Preliminaries and notation}\label{sec:notation}
$c$ will denote a a positive constant depending only on $\Omega,N,p$ that may change from line to line. We write $a\lesssim b$ if there is $c>0$ depending only on $\Omega,N,p$ such that $a\leq c b$.

Throughout $\Omega\subset\R^n$ will be a domain.
Let $1\leq p<\infty$. We denote by $L^p(\Omega)=L^p(\Omega,\R^N)$ and $W^{1,p}(\Omega)=W^{1,p}(\Omega,\R^N)$ the usual Lebesgue and Sobolev spaces respectively. $VMO(\Omega)$ denotes the space of maps with vanishing mean oscillation. Further $p'$ will denote the H\"older conjugate of $p$. We also employ the standard fractional Sobolev spaces $W^{k,p}(\Omega)$ where ${k\in (0,\infty)}$, whose theory can be found for example in \cite{Triebel1983}. We recall in particular the following fact from \cite{Guermond2017}: if $p\in(1,\infty)$, $\alpha>0$ and $\Omega$ is a regular polytope, then for $u\in W^{1+\alpha,p}(\Omega)$ the best approximation $v$ to $u$ in the space of continuous piecewise linear functions subordinate to a shape-regular triangulation of $\Omega$ of mesh-size $h$ satisfies
\begin{align}\label{eq:bestApproximation}
\|u-v\|_{W^{1,p}_\infty(\Omega)}\lesssim h^\alpha \|u\|_{W^{1+\alpha,p}_\infty(\Omega)}.
\end{align}

For $z\in \R^n$, $r>0$ we denote by $B_r(z)$ the open ball of radius $r$ around $z$.

Given vectors $u, v\in \R^n$ we denote by $|v|$ the Euclidean norm and we denote the inner product in $\R^n$ by $u\cdot v$.

Given a matrix $A\in \R^{n \times n}$ we denote by $\|A\|$ the operator norm of $A$. If $A$ is positive definite, we denote its smallest eigenvalue by $\lambda_A$ and its largest by $\Lambda_A$. In particular when $a\equiv a(x,z)\colon \Omega\times \R^{N\times n}\to \R^{N\times n}$ is such that $a(x,\cdot)$ is differentiable in $\R^{N\times n}$ for almost every $x\in\Omega$, we view $\p_z a(x,y,z)$ both as a matrix and as a linear form. If $\p_z a(x,y,z)$ is positive definite, we denote $\lambda_a = \lambda_A$, $\Lambda_a=\Lambda_A$.

We recall the definition of Muckenhoupt weights. For a fixed $1 < p <\infty$, we say that a weight $w \colon \R^n\to [0,\infty)$ belongs to $A_p$ if $w$ is locally integrable and there is a constant $C>0$ such that, for all balls $B$ in $\R^n$, we have
$$\fint_B w(x)\d x \left(\fint_B w^{-\frac q p}\d x\right)^\frac p q \leq C<\infty.$$
Given a weight $w$ on $\Omega$, we denote the weighted Lebesgue-spaces by $L^p_w(\Omega)$.

Given $\mu\geq 0$, $p\in(1,\infty)$ and $v\in \R^n$ we write $V_{\mu,p}(v) = (\mu^2+|v|^2)^\frac{p-2}{4}v$. When the choice of $p$ is clear from the context we suppress the index and write $V_{\mu}(v)=V_{\mu,p}(v)$.

We recall some well-known tools for dealing with $p$-growth. The following Lemma is standard, but the author has been unable to find a version in the literature with explicit bounds that tend to $1$ as $\gamma\to 0$. Hence a proof of the bounds of this type shown in the version below can be found in the appendix.
\begin{lemma}\label{lem:VFuncEquiv}
Let $\xi,\eta\in \R^m$ for some $m>0$. For $\gamma\geq 0$ and $\mu\geq 0$ we have
  \[
  \frac{1}{6^{\gamma}(2\gamma+1)}(\mu^2+|\eta|^2+|\eta-\xi|^2)^\gamma\leq \int_0^1 (\mu^2+|t\xi+(1-t)\eta|^2)^\gamma\d t\leq 2^\gamma(\mu^2+|\eta|^2+|\eta-\xi|^2)^\gamma
 \]
 If $\gamma\in (-1/2,0]$ we have
 \begin{align*}
   2^\gamma(\mu^2+|\eta|^2+|\eta-\xi|^2)^\gamma\leq \int_0^1 (\mu^2+|t\xi+(1-t)\eta|^2)^\gamma\d t\leq \frac{1}{4^\gamma(\gamma+1)}(\mu^2+|\eta|^2+|\xi-\eta|^2)^\gamma
 \end{align*}
\end{lemma}

We also recall that the $V_\mu$-functional enjoys a Young-type inequality.
\begin{lemma}[cf.\cite{Acerbi2002}, Lemma 2.3]\label{lem:VYoung}
Let $x,y\in \R^{N \times n}$, $\mu\geq 0$ and $1\leq p$. Let $\e>0$. Then with 
\begin{align*}
C_\e = \max\left(\frac 1 {4\e},\frac{(p-1)^{p-1}}{p^p\e^{p-1}}\right),
\end{align*}
it holds that
\begin{align*}
(\mu^2+|x|^2)^\frac{p-2}{2}x\cdot y\leq \e |V_\mu(x)|^2+
C_\e |V_\mu(y)|^2.
\end{align*}
\end{lemma}

Further we note the following estimate:
\begin{lemma}\label{lem:triangleEstimate}
For $a,b\in\R^n$ we have
\begin{align*}
\frac{3-\sqrt{5}}{2}(|a|^2+|b|^2)\leq|a|^2+|a-b|^2\leq \frac{3+\sqrt{5}}{2}(|a|^2+|b|^2)
\end{align*}
\end{lemma}

We close this section by recalling a linear algebra result from \cite{Koshelev1991}. Let $A\in \R^{n\times n}$ be positive definite. Write $A=A^+A^-$ where $A^+, A^-$ are the symmetric and skew-symmetric part of $A$, respectively. Set $C=A^+A^--A^-A^+-(A^-)^2$ and denote by $\sigma$ the largest eigenvalue of $C$. Define
\begin{align*}
K_\gamma\coloneqq \inf_{\gamma>0}\|I-\gamma A\|.
\end{align*}
Then we have
\begin{lemma}[c.f. Lemma 1.1.2 in \cite{Koshelev1991}]\label{lem:KoshelevLinearAlgebra} The optimal constant $K_\gamma$ is achieved with the following choice:
\begin{align*}
\begin{cases}
K_\gamma^2 = \dfrac{\sigma}{\sigma + \lambda_{A}^2}, \qquad \gamma = \dfrac{\lambda_{A}}{\sigma + \lambda_{A}^2} \quad& \text{ if } \sigma \geq \dfrac{\lambda_{A}(\Lambda_{A}-\lambda_{A})}{2}\\[20pt]
K_\gamma^2 = \dfrac{(\Lambda_{A}-\lambda_{A})^2+4\sigma}{(\Lambda_{A}+\lambda_{A})^2}, \quad \gamma = \dfrac{2}{\Lambda_{A}+\lambda_{A}} \quad& \text{ if } \sigma \leq \dfrac{\lambda_{A}(\Lambda_{A}-\lambda_{A})}{2}\\
\end{cases}
\end{align*}
\end{lemma}

\begin{remark}
An inspection of the proof shows that the inequality $\|I-\gamma A\|\leq K_\gamma$ still holds when $\lambda_A, \Lambda_A$ and $\sigma$ are replaced by any upper bound for the smallest and largest eigenvalue of $A$ and the largest eigenvalue of $C$, respectively, in the definition of $K_\gamma$ and $\gamma$.
\end{remark}

\section{The iterative process}\label{sec:KoshelevBasic}
We call $b\colon \Omega\times \R^{N \times n}\to\R^{N\times n}$ a reference field if the following holds: $b$ satisfies $\mathrm{\ref{ass:A1}-\ref{ass:A3}}$ and \eqref{ass:A4}. In addition, whenever $F\in L^{p'}(\Omega,\R^{N \times n})$ and $g\in W^{1,p}(\Omega,\R^N)$, there exists a unique solution $u\in W^{1,p}_g(\Omega)$ to the problem
\begin{align}\label{eq:goodOp}
\Div b(x,\D u) = \Div F.
\end{align}
\eqref{eq:goodOp} is to be understood in the sense of distributions.

\begin{remark}
The $p$-Laplacian $|z|^{p-2}z$ is an example of such a field. We encourage the reader to think of this case on a first reading.
\end{remark}

Recall the definition of the iterative process.
Let ${u_0\in W^{1,p}_g(\Omega)}$ and $f\in L^{p'}(\Omega)$. Take $\gamma>0$ to be chosen later. Define $u_{n+1}\in W^{1,p}_g(\Omega)$ to be the weak solution of the problem \eqref{eq:goodOp} with the choice
\begin{align}\label{eq:KoshIter}
F = b(x,\D u_n)-\gamma a(x,\D u_n)+ \gamma f.
\end{align}
Note that \ref{ass:A3} ensures that $F\in L^{p'}(\Omega,\R^{N \times n})$, so that the sequence $\{u_n\}$ is well-defined.

We want to show that $u_n\to u$ in $W^{1,p}(\Omega)$ where $u$ is a solution of \eqref{eq:KoshProblem}.
The crucial observation to prove convergence is the following linear algebra observation:
\begin{lemma}\label{lem:normEstimate}
 Suppose $A,B\in \mathbb{R}^{m\times m}$ are positive definite and symmetric. Then, with the choice $\gamma_* = \frac{2}{\Lambda_{B^{-1}A}+\lambda_{B^{-1}A}}$ and
 \[
  K=\frac{\Lambda_{B^{-1}A}-\lambda_{B^{-1}A}}{\Lambda_{B^{-1}A}+\lambda_ {B^{-1}A}}<1,
 \]
 the estimate
 \[
  \|B-\gamma_* A\|\leq K\|B\| 
 \]
 holds.
\end{lemma}
\begin{proof}
  By Lemma \ref{lem:KoshelevLinearAlgebra} there holds
  \begin{align*}
   \|B-\gamma A\|\leq& \|B\|\|I-\gamma B^{-1}A\|
   \leq \|B\| \frac{\Lambda_{B^{-1}A}-\lambda_{B^{-1}A}}{\Lambda_{B^{-1}A}+\lambda_{B^{-1}A}}
   = \|B\| K
  \end{align*}
  The moreover part follows immediately from Lemma \ref{lem:KoshelevLinearAlgebra}.
\end{proof}

\begin{remark}\label{rem:nonsymmetric}
Using the remark after Lemma \ref{lem:KoshelevLinearAlgebra}, whenever $a$ and $b$ are fields satisfying $\mathrm{\ref{ass:A1}-\eqref{ass:A4}}$, we can use this Lemma to ensure the existence of a $\gamma>0$ so that the inequality $\|\p_z b(x,z)-\gamma \p_z a(x,z)\|\leq K_{a,b}\|B\|$ holds uniformly in $x$ and $z$. We define
\begin{align}\label{def:Kg}
K_\gamma \coloneqq \inf \{c>0\colon \|B(x,z)-\gamma A(x,z)\|\leq c \|B(x,z)\|\text{ for almost all } x\in\Omega, \text{ every } z\in \R^{N\times n}\}.
\end{align}

In the case where $A,B$ are non-symmetric, the choice of $\gamma$ and $K$ needs to be modified with the obvious adaptions to the above proof coming from Lemma \ref{lem:KoshelevLinearAlgebra}.
With these modifications the symmetry assumption \eqref{ass:A4} may be dropped in all results of this paper.

\end{remark}

We are now able to prove Theorem \ref{thm:KoshelevBasicIntro}, which we restate for the convenience of the reader.

\begin{theorem}\label{thm:KoshelevBasic}
Suppose $b\colon \Omega \times \R^{N \times n}\to \R^N$ is a reference field. Consider $a\colon \Omega \times \R^{N \times n}\to \R^N$ satisfying $\mathrm{\ref{ass:A1}-\ref{ass:A3}}$ and \eqref{ass:A4}. If 
\begin{align*}
\begin{cases}
\dfrac{\Lambda_b}{\lambda_b}K_{a,b} \left(\dfrac{3+\sqrt{5}}{2}\right)^\frac{p-2}{2p} 6^{p-2}(p-1)< 1 &\quad\text{ for } p\geq 2 \\[20pt]
\dfrac{\Lambda_b}{\lambda_b}K_{a,b} \left(\dfrac{3-\sqrt{5}}{2}\right)^\frac{(p-2)}{4}\dfrac{2^{2-p}}{p-1}< 1 &\quad\text{ for } p\leq 2
\end{cases}
\end{align*}
then \eqref{eq:KoshProblem} has a unique solution $u\in W^{1,p}_g(\Omega)$. Moreover, there is $\gamma>0$ such that if the sequence ${u_n\in W^{1,p}_g(\Omega)}$ is generated via \eqref{eq:KoshIter}, then $u_n\to u$ in $W^{1,p}(\Omega)$. 
\end{theorem}

\begin{remark}
From the proof it will be clear that, in fact, $\{u_n\}$ generated via \eqref{eq:KoshIter} converges to the weak solution of \eqref{eq:KoshProblem} whenever $\gamma>0$ is such that 
\begin{align*}
\begin{cases}
R^{p-1} =\dfrac{\Lambda_b}{\lambda_b}K_\gamma \left(\dfrac{3+\sqrt{5}}{2}\right)^\frac{p-2}{2p} 6^{p-2}(p-1)<1 &\quad\text{ if } p\geq 2 \\
R^{p/2}=\dfrac{\Lambda_b}{\lambda_b}K_\gamma \left(\dfrac{3-\sqrt{5}}{2}\right)^\frac{(p-2)}{4}\dfrac{2^{2-p}}{p-1}< 1 &\quad \text{ if } p\leq 2.
\end{cases}
\end{align*}
Here $K_\gamma$ is the constant from \eqref{def:Kg}.
Further, the convergence occurs at linear rate $R$. To be precise, the following estimates hold:
\begin{align*}
\|u_{n+1}-u_n\|_{W^{1,p}(\Omega)}\leq& R^n \|u_1-u_0\|_{W^{1,p}(\Omega)}\\
\|u_{n+1}-u\|_{W^{1,p}(\Omega)}\leq& R^n \|u_1-u\|_{W^{1,p}(\Omega)},
\end{align*}
where $u$ is the weak solution of \eqref{eq:KoshProblem}.
\end{remark}

Before presenting the proof we want to briefly outline the main idea. Consider the sequence $(u_n)$ generated via \eqref{eq:KoshIter}. Subtracting the equations defining $u_{n+1}$ and $u_n$ we find
\begin{align*}
\Div\left( b(x,\D u_{n+1})-b(x,\D u_n)\right)=\Div \left(b(x,\D u_n)-b(x,\D u_{n-1})-\gamma\left(a(x,\D u_n)-a(x,\D u_{n-1})\right)\right)
\end{align*}
Using the mean-value theorem and denoting $A(x,z)=\p_z a(x,z)$, $B(x,z)=\p_z b(x,z)$, we may rewrite this as
\begin{align*}
\Div \left(B(x,\xi_{n+1})\D(u_{n+1}-u_n)\right) = \Div \left((B-\gamma A)(x,\tilde\xi_n)\D(u_n-u_{n-1})\right).
\end{align*}
for some $\xi_{n+1}$ lying on the line segment between $\D u_{n+1}$ and $\D u_n$ and $\tilde \xi_n$ lying on the line segment between $\D u_n$ and $\D u_{n-1}$. If $\|B(x,z)-\gamma A(x,z)\|$ is sufficiently small, uniformly in $x$ and $z$, then applying the ellipticity assumption to bound the left-hand side from below, we are able to show that for some $R<1$,
$$
\|\D(u_{n+1}-u_n)\|_{L^p(\Omega)}\leq R\|\D(u_n-u_{n-1})\|_{L^p(\Omega)}.
$$
We then conclude easily.
\begin{proof}
Let $\gamma$ be as in Lemma \ref{lem:normEstimate}, so that, using the notation of the lemma, $K_\gamma = K_{a,b}$.
Throughout the proof we write $A(x,z)= \p_z a(x,z)$ and $B(x,z)=\p_z b(x,z)$.
 
 Let $n\geq 1$.
Test the equations defining $u_{n+1}$ and $u_n$ against $u_{n+1}-u_n$ to find:
 \begin{align*}
  &\int_\Omega \left(b(x,\D u_{n+1})-b(x,\D u_n)\right)\D (u_{n+1}-u_n)\d x\\
  =&\int_\Omega \left(b(x,\D u_n)-b(x,\D u_{n-1})\right)\D (u_{n+1}-u_n)-\gamma \left(a(x,\D u_n)-a(x,\D u_{n-1})\right)\D (u_{n+1}-u_n)\d x\\
  =&\int_\Omega \int_0^1 \left(B(x,z_n(\theta))-\gamma A(x,z_n(\theta))\right)\D (u_n-u_{n-1})\cdot \D (u_{n+1}-u_n)\ d\theta\d x.
 \end{align*}
 Here $z_n(\theta) = (1-\theta) \D u_{n-1}+ \theta \D u_n$. The last line follows from an application of the mean value theorem. Now proceed to estimate both sides of this equality. 
 
 We focus first on the case $p\geq 2$.
Using again the mean-value theorem and writing $z_{n+1}(\theta) = (1-\theta) \D u_{n+1}+ \theta \D u_n$,  by assumption \ref{ass:A2}, the left-hand side gives:
  \begin{align*}
   &\int_\Omega \left(b(x,\D u_{n+1})-b(x,\D u_n)\right)\D (u_{n+1}-u_n)\ d x\\
  =& \int_\Omega \int_0^1 (B(x,z_{n+1}(\theta))\D (u_{n+1}-u_n)\cdot \D (u_{n+1}-u_n)\d \theta\d x\\
   \geq& \lambda_b \int_\Omega\int_0^1 (\mu^2+|z_{n+1}(\theta)|^2)^\frac{p-2}{2}|\D (u_{n+1}-u_n)|^2\d \theta\d x\\
   \geq&   \frac{\lambda_b}{6^\frac{p-2}{2}(p-1)} \int_\Omega (\mu^2+|\D u_{n+1}|^2+|\D u_{n+1}-\D u_n|^2)^\frac{p-2}{2}|\D (u_{n+1}-u_n)|^2\d x
  \end{align*}
  where the last line uses Lemma \ref{lem:VFuncEquiv}.
  
  Note that $\|B(x,w)\|\leq\Lambda_b(\mu^2+|w|^2)^\frac{p-2}{2}$ by \ref{ass:A2}. Hence 
  on the right-hand side by Lemma \ref{lem:normEstimate} and Lemma \ref{lem:triangleEstimate},
  \begin{align*}
   &\int_\Omega \int_0^1 \left(B(x,z_n(\theta))-\gamma A(x,z_n(\theta))\right)\D (u_n-u_{n-1})\cdot \D (u_{n+1}-u_n)\d \theta\d x\\
   \leq& \Lambda_b K\int_\Omega \int_0^1 (\mu^2+|z_n(\theta)|^2)^\frac{p-2}{2} |\D u_n-\D u_{n-1}||\D u_{n+1}-\D u_n|\d\theta\d x\\
   \leq& \Lambda_b K \int_\Omega (\mu^2+|\D u_n|^2+|\D u_n-\D u_{n-1}|^2)^\frac{p-2}{2}|\D (u_n-u_{n-1})||\D (u_{n+1}-u_n)|\d x =I.
  \end{align*}
  We now apply Lemma \ref{lem:VYoung} and Lemma \ref{lem:VFuncEquiv} to find for any $\e>0$,
  \begin{align*}
 I    \leq& \Lambda_b K \left(C_\e\int_\Omega V_{(\mu^2+|\D u_n|^2)^{1/2}}(\D (u_{n+1}-u_{n}))^2+\e \int_\Omega V_{(\mu^2+|\D u_n|^2)^{1/2}}(\D (u_n-u_{n-1}))^2\right)\\
 \leq& \Lambda_b K \Big(\e\int_\Omega V_{(\mu^2+|\D u_n|^2)^{1/2}}(\D (u_n-u_{n-1}))^2\\
 &+C_\e \left(\frac{3+\sqrt{5}}{2}\right)^\frac{p-2}{2}\int_\Omega (\mu^2+|\D u_{n+1}|^2+|\D u_{n+1}-\D u_n|^2)^\frac{p-2}{2}|\D u_{n+1}-\D u_n|^2\Big).
  \end{align*}

  Combining these two estimates and re-arranging gives:
  \begin{align*}
   &\int_\Omega (\mu^2+|\D u_{n+1}|^2+|\D u_{n+1}-\D u_n|^2)^\frac{p-2}{2}|\D (u_{n+1}-u_n)|^2\d x\\
   \leq& \frac{\e C_1}{C_0-C_\e C_2} \int_\Omega (\mu^2+|\D u_n|^2+|\D (u_{n-1}-u_n)|^2)^\frac{p-2}{2}|\D (u_{n+1}-u_n)|^2\d x.
  \end{align*}
  where $C_0 = \frac{\lambda_b}{6^\frac{p-2}{2}(p-1)}$, $C_1 = \Lambda_b K$ and $C_2 = (\frac{3+\sqrt{5}}{2})^\frac{p-2}{2}\Lambda_b K$.

  Optimising in $\varepsilon$ and using the hypothesis this shows by induction that
  \[
   \int_\Omega (|\D u_n|^2+|\D (u_{n+1}-u_n)|^2)^\frac{p-2}{2}|\D (u_{n+1}-u_n)|^2\d x\to0
  \]
  at linear rate.

  As $p\geq2$, it immediately follows that $\{u_n\}$ converges in $W^{1,p}(\Omega)$ at linear rate. Necessarily the limit $u$ lies in $W^{1,p}_g(\Omega)$ and it is a solution of problem \eqref{eq:KoshProblem}. Moreover considering any solution $v\in W^{1,p}_g(\Omega)$ of \eqref{eq:KoshProblem} and considering the above estimates with the starting point
\begin{align*}
&\int_\Omega (b(x,\D u_{n+1})-b(x,\D v))\D (u_{n+1}-v) \d x\\
=& \int_\Omega (b(x,\D u_n)-b(x,\D v)
-\gamma (a(x,\D u_n)-a(x,\D v))+\gamma f)\cdot \D (u_{n+1}-v)\d x
\end{align*}  
we find that, for any  $u_0\in W^{1,p}_g(\Omega)$, the corresponding iterative process converges to $v$ in $W^{1,p}_g(\Omega)$. In particular, the solution to \eqref{eq:KoshProblem} is unique.
  
We now consider $p<2$. Proceed as before to obtain
  \begin{align*}
   &\int_\Omega \left(b(x,\D u_{n+1})-b(x,\D u_n)\right)\D (u_{n+1}-u_n)\d x\\
   \geq&   \lambda_b \int_\Omega (\mu^2+|\D u_{n+1}|^2+|\D u_{n+1}-\D u_n|^2)^\frac{p-2}{2}|\D (u_{n+1}-u_n)|^2\d x
   \end{align*}
   and
    \begin{align*}
       &\int_\Omega \int_0^1 \left(B(x,z(\theta))-\gamma A(x,z(\theta))\right)\D (u_n-u_{n-1})\cdot \D (u_{n+1}-u_n)\d\theta\d x\\
   \leq& \frac{\Lambda_b K}{2^{p-2}(p-1)} \Bigg(\left(\frac{3-\sqrt{5}}{2}\right)^\frac{p-2}{2}C_\e\int_\Omega (\mu^2+|\D u_{n+1}|^2+|\D u_{n+1}-\D u_n|^2)^\frac{p-2}{2}|\D (u_n-u_{n-1})|^2\d x \\
 &+ \e \int_\Omega (\mu^2+|\D u_n|^2+|\D u_n-\D u_{n-1}|^2)^\frac{p-2}{2}|\D (u_n-u_{n-1})|^2\d x\Bigg).
  \end{align*}
 Combining these estimates and our assumptions, we again find that   \[
   \int_\Omega (|\D u_n|^2+|\D u_{n+1}|^2)^\frac{p-2}{2}|\D (u_{n+1}-u_n)|^2\d x\to0
  \]
  at linear rate.

  Note that by H\"older there is $c>0$ such that
  \begin{align*}
   &\int_\Omega (|\D u_n|+|\D (u_{n+1}-u_n)|)^{p-2}|\D (u_{n+1}-u_n)|^2\d x\\
   \geq& c \frac{\|\D (u_{n+1}-u_n)\|_p^2}{\|\D u_n\|_p^{2-p}+\|\D u_{n+1}\|_p^{2-p}}.
  \end{align*}
  If $\{\D u_n\}$ is a bounded sequence in $W^{1,p}$, convergence of $\{u_n\}$ in $W^{1,p}(\Omega)$ and existence and uniqueness of solutions to \eqref{eq:KoshProblem} follow by repeating the arguments of the case $p\geq 2$. For clarity of presentation and as it essentially follows by repeating arguments similar to those of this proof we postpone the proof of boundedness to Lemma \ref{cor:boundedness}.
  \end{proof}

\begin{lemma}\label{cor:boundedness}
Suppose the assumptions of Theorem \ref{thm:KoshelevBasic} hold. Then $\{u_n\}$ is bounded in $W^{1,p}(\Omega)$.
\end{lemma}
\begin{proof}

Let $\gamma$ be as in Lemma \ref{lem:normEstimate}, so that, using the notation of the lemma, $K_\gamma = K_{a,b}$. We write $A(x,z)=\p_z a(x,z)$ and $B(x,z)=\p_z(x,z)$.
 
 Let $n\geq 1$.
Test the equation defining $u_{n+1}$ against $u_{n+1}$ to find:
 \begin{align*}
  &\int_\Omega \left(b(x,\D u_{n+1})-b(x,0)\right)\D u_{n+1}\d x\\
  =&\int_\Omega \left(b(x,\D u_n)-b(x,0)\right)\D u_{n+1}\\
  &-\gamma \left(a(x,\D u_n)-a(x,0)\right)\D u_{n+1}+\gamma a(x,0)\D u_{n+1}\ d x\\
  =&\int_\Omega \int_0^1 \left(B(x,z_n(\theta))-\gamma A(x,z_n(\theta))\right)\D u_n\cdot \D u_{n+1}+\gamma a(x,0)\D u_{n+1}\d \theta\d x.
 \end{align*}
 Here $z_n(\theta) = \theta u_n$. The last line follows from an application of the mean value theorem. Now proceed to estimate both sides of this equality exactly as in Theorem \ref{thm:KoshelevBasic}.
 
 We focus first on the case $p\geq 2$.
Using again the mean-value theorem and writing $z_{n+1}(\theta) = \theta u_{n+1}$,  by assumption \ref{ass:A2}, the left-hand side gives:
  \begin{align*}
   \int_\Omega \left(b_\alpha(x,\D u_{n+1})-b_\alpha(x,0)\right)\D u_{n+1}\d x
  =& \int_\Omega \int_0^1 (B(x,z_{n+1}(\theta))\D (u_{n+1})\cdot \D u_{n+1}\d \theta\d x\\
   \geq& \lambda_b \int_\Omega\int_0^1 (\mu^2+|z_{n_1}(\theta)|^2)^\frac{p-2}{2}|\D u_{n+1}|^2\d \theta\d x\\
   \geq&   \frac{\lambda_b}{(p-1)6^\frac{p-2}{2}}\int_\Omega (\mu^2+|\D u_{n+1}|^2)^\frac{p-2}{2}|\D u_{n+1}|^2\d x
  \end{align*}
  where the last line uses Lemma \ref{lem:VFuncEquiv}.
  
  Note that $\|B(x,z)\|\leq\Lambda_b(\mu^2+|z|^2)^\frac{p-2}{2}$ by \ref{ass:A2}. Hence 
  on the right-hand side by Lemma \ref{lem:normEstimate} and Lemma \ref{lem:VFuncEquiv}, as well as \ref{ass:A3}
  \begin{align*}
   &\int_\Omega \int_0^1 \left(B(x,z_n(\theta))-\gamma A(x,z_n(\theta))\right)\D u_n\cdot \D u_{n+1}+\gamma a(x,0)\cdot \D u_{n+1}\d \theta\d x\\
   \leq& \Lambda_b K\int_\Omega \int_0^1 (\mu^2+|z_n(\theta)|^2)^\frac{p-2}{2} |\D u_n||\D u_{n+1}|+C\gamma |\D u_{n+1}|\mathrm{d\theta}\d x\\
   \leq& \Lambda_b K \int_\Omega (\mu^2+|\D u_n|^2)^\frac{p-2}{2}|\D u_n||\D (u_{n+1})|\d x + C\gamma \int_\Omega |\D u_{n+1}|\d x = I + II.
  \end{align*}
  We now apply Lemma \ref{lem:VYoung}, to find
  \begin{align*}
 I    \leq& \Lambda_b K \left(C_\e\int_\Omega V_{\mu}(\D u_{n+1})^2+\e \int_\Omega V_{\mu}(\D u_n)^2\right).
  \end{align*}
  Young's inequality also gives
  \begin{align*}
  II\leq& C\gamma\|\D u_{n+1}\|_{L^p(\Omega)}\leq \e_1\|\D u_{n+1}\|_{L^p(\Omega)}^p + C(\e_1).
  \end{align*}

  Combining these two estimates and re-arranging gives:
  \begin{align*}
   &\int_\Omega (\mu^2+|\D u_{n+1}|^2)^\frac{p-2}{2}|\D u_{n+1}|^2\d x\\
   \leq& \frac{\e C_1}{C_0-C_\e C_2} \int_\Omega (\mu^2+|\D u_n|^2)^\frac{p-2}{2}|\D u_n)|^2\d x+\frac{1}{C_0-\e C_2}(C(\e_1)+\e_1\|\D u_{n+1}\|_{L^p(\Omega)}^p).
  \end{align*}
  where $C_0 = \frac{\lambda_b}{6^\frac{p-2}{2}(p-1)}$ and $C_1 = C_2 = \Lambda_b K$.

  Optimising in $\varepsilon$, we can ensure that
  $\frac{\e C_1}{C_0-C_\e C_1}<1$. Now choosing $\e_1$ sufficiently small, we conclude that
  \begin{align}\label{eq:boundedFinal}
  \|V_\mu(\D u_{n+1})\|_{L^2(\Omega)}^2\leq \eta\|V_\mu(\D u_{n+1})\|_{L^2(\Omega)}^2 + C
  \end{align}
  for some constant $\eta<1$. The conclusion follows using induction.

For $p\leq 2$ we argue similarly to obtain \eqref{eq:boundedFinal}
with the choices $C_0 = \lambda_b$, $C_1 = \Lambda_b K$, $C_2 = \Lambda_b K (p-1)^{-1}4^{(2-p)2}$. Optimising in $\e$ we can again ensure that $$\dfrac{\e C_1}{C_0-C_\e C_1}<1$$ and use this to conclude as before.
\end{proof}

\section{Numerical analysis and experiments}\label{sec:numerical}

Throughout this section we assume that $a(x,z), b(x,z)$ satisfy $\mathrm{\ref{ass:A1}-\ref{ass:A3}}$ and \eqref{ass:A4}. We further assume that the assumptions of Theorem \ref{thm:KoshelevBasic} hold. In particular, this will imply that, with an appropriate choice of $\gamma$, the iterative process \eqref{eq:KoshIter} converges in $W^{1,p}(\Omega)$ to a solution of \eqref{eq:KoshProblem} and so, in particular, the iterates $\{u_n\}$ are uniformly bounded in $W^{1,p}(\Omega)$. We fix such a choice of $\gamma$ from now on. 

Recall that we study the following numerical scheme:
Suppose $\Omega$ is a regular polytope. Consider a sequence of shape-regular triangulations $\{T_h\}_{h\in(0,1]}$ of $\Omega$ of mesh-size $h$. Denote by $X_h$ the space of continuous piecewise linear (subordinate to $T_h$) functions. 
Choose $u_0\in X_h$. Define $u_{n+1}^h\in X_h$ inductively as a solution of the problem:
\begin{align}\label{eq:numScheme}
\begin{cases}
\int_\Omega b(x,\D u_{n+1}^h)\cdot \D \phi\d x = \int_\Omega (b(x,\D u_n^h)-\gamma a(x,\D u_n^h)+\gamma f)\cdot \D \phi\d x &\quad \forall \phi\in X_h. \\
u_{n+1} = 0 &\quad\text{ on } \p\Omega
\end{cases}
\end{align}
Note that due to $\mathrm{\ref{ass:A1}-\ref{ass:A3}}$ and the theory of monotone operators \eqref{eq:numScheme} is well-defined.

\subsection{Analysis of the numerical scheme}
In this section we present our main results regarding convergence, a-priori and a-posteriori estimates for \eqref{eq:numScheme}. We begin with the following convergence result:
\begin{theorem}\label{thm:KoshelevNum}
Let $u_0\in X_h$ and consider the sequence $\{u_n^h\}$ generated by \eqref{eq:numScheme} starting from $u_0$. Then $u_n^h\to u_h$ in $W^{1,p}(\Omega)$ where $u_h\in X_h$ solves
\begin{align}\label{eq:numProblem}
\begin{cases}
\int_\Omega a(x,\D u_h)\cdot \D \phi\d x = \int_\Omega f\cdot \D \phi\d x &\quad\forall \phi\in X_h\\
u_h = 0 &\quad\text{ on } \p\Omega.
\end{cases}
\end{align}
Moreover there is $0<C_1<1$ such that
\begin{align*}
\|u_{n+1}^h-u_n^h\|_{W^{1,p}(\Omega)}\leq C_1^n\|u_1-u_0\|_{W^{1,p}(\Omega)}
\end{align*}
\end{theorem}
\begin{proof}
This follows from carrying out line by line the proof of Theorem \ref{thm:KoshelevBasic}, replacing $\{u_n\}$ with $\{u_n^h\}$ and $u$ with $u_h$ respectively.
\end{proof}

We now wish to show that $u_h\to u$ in $W^{1,p}(\Omega)$ as $h\to 0$ where $u\in W^{1,p}_0(\Omega)$ solves \eqref{eq:KoshProblem}. In order to prove this we make the following regularity assumption on $b$:
\textit{If $F\in L^{p'}(\Omega)$ then there is $\alpha>0$ such that \eqref{eq:goodOp} has a solution $u$ in $W^{1+\alpha,p}(\Omega)$ and moreover we have the following estimate for some $c>0$:}
\begin{align}\label{eq:addRegularity}
\|u\|_{W^{1+\alpha,p}(\Omega)}\leq c\left(1+\|u\|_{W^{1,p}(\Omega)}+\|F\|_{L^{p'}(\Omega)}^{1/(p-1)}\right)
\end{align}

\begin{remark}
\eqref{eq:addRegularity} is satisfied for example when $b(x,z)=|z|^{p-2}z$ with any choice of $\alpha>0$ such that ${\alpha < \min\left(1/(p-1)^2,(p-1)^2\right)}$, see \cite{Simon1981}.
\end{remark}

We proceed to study the effect of decreasing the mesh-size $h$. We use the notation of Theorem \ref{thm:KoshelevNum}.
\begin{theorem}\label{thm:numericalConvergence}
Set $g=0$.
Assume the assumptions of Theorem \ref{thm:KoshelevNum} hold. Suppose moreover that \eqref{eq:addRegularity} is satisfied. Choose $u_0\in X_h$ and let $\{u_n^h\}$ be the sequence generated by \eqref{eq:numScheme}. Suppose $u\in W^{1,p}_0(\Omega)$ solves \eqref{eq:KoshProblem}. Then $u\in W^{1+\alpha,p}(\Omega)$ and $u^h\to u$ in $W^{1,p}(\Omega)$ as $h\to 0$. Moreover 
 we have the estimate \begin{align*}
\|u_{n+1}^h-u\|_{W^{1,p}(\Omega)}\leq c h^\frac{2\alpha}{\max(2,p)}c\left(\|f\|_{L^{p'}(\Omega)}\right)+C_1^n\|u_1-u_0\|_{W^{1,p}(\Omega)}.
\end{align*}
\end{theorem}

\begin{proof}
Consider the iterative process $\{u_n\}$ started from $0$. By Theorem \ref{thm:KoshelevBasic}, $u_n\to u$ in $W^{1,p}(\Omega)$. Using \eqref{eq:addRegularity}, we find
\begin{align*}
\|u_n\|_{W^{1+\alpha,p}(\Omega)}\leq c\left(1+\|u_n\|_{W^{1,p}(\Omega)}+\|b(x,\D u_n)-\gamma a(x,\D u_n)+\gamma f\|_{L^{p'}(\Omega)}^{1/(p-1)}\right)
\end{align*}
Repeating arguments from the proof of Lemma \ref{cor:boundedness} we deduce using induction that there is $c(\|f\|_{L^{p'}(\Omega)})$ such that for all $n\geq 0$,
\begin{align*}
\|u_n\|_{W^{1+\alpha,p}(\Omega)}\leq c(\|f\|_{L^{p'}(\Omega)}).
\end{align*}
Extracting a weakly convergent subsequence we conclude the same estimate holds for $u$.

Let $v$ be the best approximation to $u$ in $X_h$. Using \ref{ass:A2}, the fact that $u,u_h$ solve \eqref{eq:KoshProblem} and \eqref{eq:numProblem} respectively we find
\begin{align*}
&\lambda_a\int_\Omega (\mu^2+|\D u|^2+|\D u^h|^2)^\frac{p-2}{2}|\D u-\D u^h|^2\d x\\
 \leq& \int_\Omega (a(x,\D u)-a(x,\D u^h))\D (u-u^h)\d x\\
=& \int_\Omega (a(x,\D u)-a(x,\D u^h))\D (u-v)\d x\\
\lesssim& \Lambda_a \int_\Omega (\mu^2+|\D u|^2+|\D (u-u^h)|^2)^\frac{p-2}{2}|\D (u-u^h)||\D (u-v)|\d x\\
\lesssim& C_\e\Lambda_a \int_\Omega (\mu^2+|\D u|^2+|\D u_h|^2)^\frac{p-2}{2}|\D (u-u_h)|^2\d x \\
&\quad+\e \Lambda_a \int_\Omega (\mu^2+|\D u|^2+|\D v|^2)^\frac{p-2}{2}|\D (u-v)|^2\d x
\end{align*}
where to obtain the last line we have used Lemma \ref{lem:VYoung}. Choosing $\e$ sufficiently large we conclude using \eqref{eq:bestApproximation}
\begin{align}\label{eq:DiffFinalNum}
I=&\int_\Omega (\mu^2+|\D u|^2+|\D u^h|^2)^\frac{p-2}{2}|\D (u-u^h)|^2\nonumber\\
\lesssim& \int_\Omega (\mu^2+|\D u|^2+|\D v|^2)^\frac{p-2}{2}|\D (u-v)|^2\d x\nonumber\\
\leq& \|\D (u-v)\|_{L^p(\Omega)}^2\left(1+\|\D u\|_{W^{1,p}(\Omega)}^{p-1}+\|\D (u-v)\|_{W^{1,p}(\Omega)}^{p-1}\right)\nonumber\\
\lesssim& h^{2\alpha} \|u\|_{W^{1+\alpha,p}(\Omega)}\left(1+\|u\|_{W^{1,p}(\Omega)}^{p-1}+h^{\alpha(p-1)}\|u\|_{W^{1+\alpha,p}(\Omega)}^{p-1}\right)
\end{align}
If $p\geq 2$, $I\geq \|\D (u-u_h)\|_{L^p(\Omega)}^p$ whereas if $p\leq 2$, we have by applying H\"older's inequality ${I\geq \|\D (u-u^h)\|_{L^p(\Omega)}^2 (1+\|\D u\|_{L^p(\Omega)}+\|\D u^h\|_{L^p(\Omega)})}$. Recalling the standard estimates $\|\D u\|_{L^p(\Omega)}\leq \|f\|_{L^{p'}(\Omega)}^{1/(p-1)}$ and $\|\D u^h\|_{L^p(\Omega)}\leq \|f\|_{L^{p'}(\Omega)}^{1/(p-1)}$, we conclude by combining \eqref{eq:DiffFinalNum}, Theorem \ref{thm:KoshelevNum} and the inequality
\begin{align*}
\|u_{n+1}^h-u\|_{W^{1,p}(\Omega)}\leq \|u_{n+1}^h-u^h\|_{W^{1,p}(\Omega)}+\|u^h-u\|_{W^{1,p}(\Omega)}.
\end{align*}
\end{proof}

We also have an a-posteriori error bound. The proof follows \cite{Heid2020}.
\begin{prop}
Assume the conditions of Theorem \ref{thm:KoshelevNum} hold. Then we have
\begin{align*}
\|\D (u_n^h-u^h)\|_{L^p(\Omega)}\leq C(\|f\|_{L^{p'}(\Omega)},\|\D u_0\|_{L^p(\Omega)})\|\D (u_n^h-u_{n-1}^h)\|_{L^p(\Omega)}^\frac{2}{\max(2,p)}.
\end{align*}
\end{prop}
\begin{proof}
We compute using \ref{ass:A2} and Lemma \ref{lem:VYoung},
\begin{align*}
&\lambda\gamma\int_\Omega (\mu^2+|\D (u^h-u_{n-1}^h)|^2+|\D u_{n-1}^h|^2)^\frac{p-2}{2}|\D u^h-\D u_{n-1}^h|^2\d x\\
\lesssim& \gamma\int_\Omega (a(x,\D u^h)-a(x,\D u_{n-1}^h))\cdot \D (u^h-u_{n-1}^h)\d x\\
=& \int_\Omega \gamma f\cdot \D (u^h-u_{n-1}^h) + (b(x,\D (u_n^h))-b(x,\D u_{n-1}^h))-\gamma f)\cdot \D (u^h-u_{n-1}^h)\d x\\
\lesssim&\int_\Omega \left(\mu^2+|\D u_{n-1}^h|^2+|\D (u_n^h-u_{n-1}^h)|^2\right)^\frac{p-2}{2}|\D (u_n^h-u_{n-1}^h)| |\D (u^h-u_{n-1}^h)|\d x\\
\leq& \e\int_\Omega \left(\mu^2+|\D u_{n-1}^h|^2+|\D (u_n^h-u_{n-1}^h)|^2\right)^\frac{p-2}{2}|\D (u_n^h-u_{n-1}^h)|^2\d x\\
&\quad +C_\e \int_\Omega (\mu^2+|\D (u^h-u_{n-1}^h)|^2+|\D u_{n-1}^h|^2)^\frac{p-2}{2}|\D u^h-\D u_{n-1}^h|^2\d x
\end{align*}
Thus choosing $\e$ sufficiently large, re-arranging, employing by now standard arguments and recalling that $\{u_n^h\}$ is bounded uniformly in $W^{1,p}(\Omega)$ we obtain
\begin{align*}
\|\D (u^h-u_{n-1}^h)\|_{L^p(\Omega)}\lesssim c\left(\|\D u^h\|_{L^p(\Omega)},\|\D u_0\|_{L^p(\Omega)}\right)\|\D (u_n^h-u_{n-1}^h)\|_{L^p(\Omega)}^\frac{2}{\max(2,p)}.
\end{align*}

By the triangle inequality, and using $\|u^h\|_{W^{1,p}(\Omega)}\lesssim \|f\|_{L^{p'}(\Omega)}^{1/(p-1)}$, we conclude the desired estimate:
\begin{align*}
\|\D (u_n^h-u^h)\|_{L^p(\Omega)}\leq& \|\D (u_{n-1}^h-u^h)\|_{L^p(\Omega)}+\|\D (u_n^h-u_{n-1}^h)\|_{L^p(\Omega)}\\
\leq& c\left(\|f\|_{L^{p'}(\Omega)},\|\D u_0\|_{L^p(\Omega)}\right)\|\D (u_n^h-u_{n-1}^h)\|_{L^2(\Omega)}.
\end{align*}
\end{proof}

We close this section by detailing a modification of the numerical scheme, following the algorithm outlined in \cite{Heid2020}. We strengthen our assumptions on $\{T_h\}$ and assume $\{T_k\}_{k\in \N}$ is a sequence of shape-regular triangulations with meshsize $h_k\to 0$ as $k\to\infty$. Moreover, we assume that $\{T_k\}$ is obtained from $\{T_{k-1}\}$ by refinement. Denote by $X^k$ the space of continuous piecewise linear functions subordinate to $\{T_k\}$.
We then consider the following algorithm:
\begin{algorithm}
\begin{algorithmic}
\STATE Set $u_0^0=0$. Set $k=0$. Choose a maximal mesh $k_{max}$ and a sequence of tolerances $\tp{tol}_k\to 0$ as $k\to\infty$.
\REPEAT
\WHILE{$\|u_n^k-u_{n-1}^k\|> \tp{tol}_k$}
\STATE solve $-\Delta u^k_{n+1} = -\Delta u^k_n +\gamma \Div (a(x,\D u^k_n)-f)$ in $X^k$
\STATE $n \leftarrow n+1$
\ENDWHILE
\STATE $u_0^{k+1}\leftarrow u_n^k$
\STATE $k\leftarrow k+1$
\UNTIL $k = k_{max}$.
\end{algorithmic}
\caption{A variant of the iteration scheme}
\label{alg:alternative}
\end{algorithm}

From the results of this section it is clear that under the assumptions of Theorem \ref{thm:numericalConvergence} Algorithm \ref{alg:alternative} converges to a solution of \eqref{eq:KoshProblem} as $k_{max}\to\infty$.

\subsection{Numerical experiments}
The results presented in this section are obtained using Firedrake \cite{petsc-user-ref,petsc-efficient,Dalcin2011,Rathgeber2016,MUMPS01,MUMPS02}.
 Throughout this section $\Omega = [0,1]^3$. We consider shape-regular triangulations of $\Omega$ with uniformly spaced nodes at distance $h =2^{-i}$. Note that in this set-up the mesh-size is $\sqrt{3}h$.

\textbf{A linear example:} We first consider a linear example where the exact solution is known. We choose, $$f = \left(\begin{matrix}8\pi & 8\pi & 8\pi\\ 10\pi& 10\pi & 10\pi \\ 2\pi& 2\pi & 2\pi\end{matrix}\right)(v,v,v)^T \quad\text{ where } v = \left(\begin{matrix}cos(2\pi x)sin(2\pi y)sin(2\pi z) \\ sin(2\pi x)cos(2\pi y) sin(2\pi z) \\ sin(2\pi x)sin(2\pi y)cos(2\pi z)\end{matrix}\right)$$
and consider the problem
\begin{align*}
\begin{cases}-\Div A\,\D u = \Div f \quad&\text{ in } \Omega\\
u = 0 &\text{ on } \p\Omega
\end{cases}\qquad \text{ where } A = \left(\begin{matrix} 1 & 1 & 2 \\  0 & 2 & 3 \\ 0 & 0 & 1
\end{matrix}\right).
\end{align*}
Note that with $\tilde u(x,y,z)=sin(2\pi x)sin(2\pi y)sin(2\pi z)$, $$u(x,y,z)=(\tilde u(x,y,z),\tilde u(x,y,z),\tilde u(x,y,z))^T$$ is the exact solution of this problem. 

For the iteration scheme we choose $u_0 = 0$, $b(x,z)=z$ and $\gamma = 2/3$ or $\gamma = 1/2$. We solve each iteration step using GMRES with an incomplete LU factorisation to precondition the problem. The iteration is terminated when $\|u_{n+1}-u_n\|_{H^1(\Omega)}\leq 10^{-9}$.

We record the $H^1$-error of the numerical solution $v$ computed using the iterative scheme \eqref{eq:numScheme} in Table \ref{tab:linearExampleConv}. We also record the number of iterations needed.
\begin{table}[ht]
\begin{center}
\begin{tabular}{ c|c|c|c|c|c }
$h$  & $2^{-1}$ & $2^{-2}$ & $2^{-3}$ & $2^{-4}$ & $2^{-5}$ \\
\hline
$\gamma = 2/3$, error & $6.6924$ & $5.3767$ & $3.2456$ & $1.6786$ & $0.8436$\\
\hline
$\gamma = 2/3$, iterations & $1$ & $22$ & $22$ & $22$ &  $22$\\
\hline
$\gamma = 1/2$, error & $6.6924$ & $5.3767$ & $3.2456$ & $1.6786$ & $0.8436$\\
\hline
 $\gamma = 1/2$, iterations & $1$ & $33$ & $33$ & $34$ & $34$
\end{tabular}
\end{center}
\caption{$\|u-v\|_{H^1(\Omega)}$}\label{tab:linearExampleConv}
\end{table}

We also compute the numerical solution $u^h$ directly using a $LU$-factorisation. We record the $H^1$-distance between $u_n^h$ and $u^h$ for two different choices of $\gamma$ in Fig.\ref{fig:linear}.
\begin{figure}[h!]
  \centering
  \begin{subfigure}[b]{0.4\linewidth}
    \caption{$\gamma = 2/3$}
        \includegraphics[width=\linewidth]{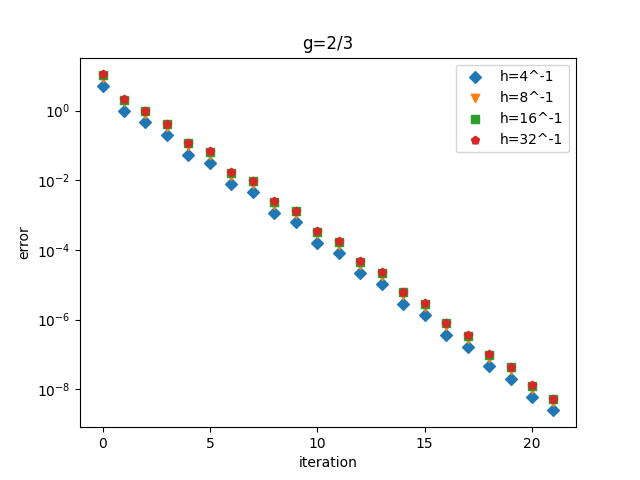}
    \caption{$\gamma = 2/3$}
  \end{subfigure}
  \begin{subfigure}[b]{0.4\linewidth}
    \includegraphics[width=\linewidth]{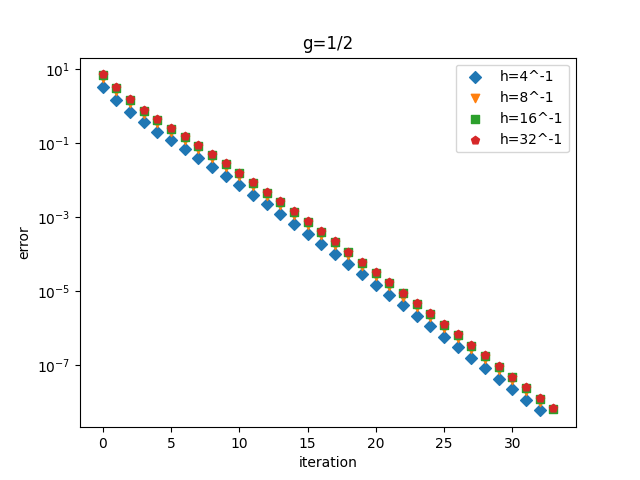}
    \caption{$\gamma = 1/2$}
  \end{subfigure}
  \caption{$\|u_n^h-u_h\|_{H^1(\Omega)}$}
\label{fig:linear}
\end{figure}
 
\textbf{A nonlinear example}
We also consider the nonlinear problem
\begin{align*}
\begin{cases}
-\Div (1+|\D u|^4) A\D u + |u|^4 u = (y,x^2,z^2+x^2)^T\quad &\text{ in } \Omega \\
 u = 0 &\text{ on } \p\Omega 
 \end{cases} \qquad \text{with } A=\left(\begin{matrix} 1 & 3 & 5 \\  0 & 2 & 4 \\ 0 & 0 & 1\end{matrix}\right).
 \end{align*}
We set $b(x,z)=(1+|z|^4)z$, $\gamma = 0.65$ and $u_0 = 0$ and employ \eqref{eq:numScheme}. Each (non-linear) iteration step is solved using the standard 'solve'-method in Firedrake. This utilises a nonlinear Newton linesearch scheme where the linear step is computed using GMREs. Denote the solution obtained in this way with $u_{dir}$. We record the $H^1$-distance $\|u_{dir}-u_n^h\|_{H^1(\Omega)}$ in Fig.\ref{fig:nonlinear}.

\begin{figure}[h!]
\begin{center}
\includegraphics[width = \linewidth/2]{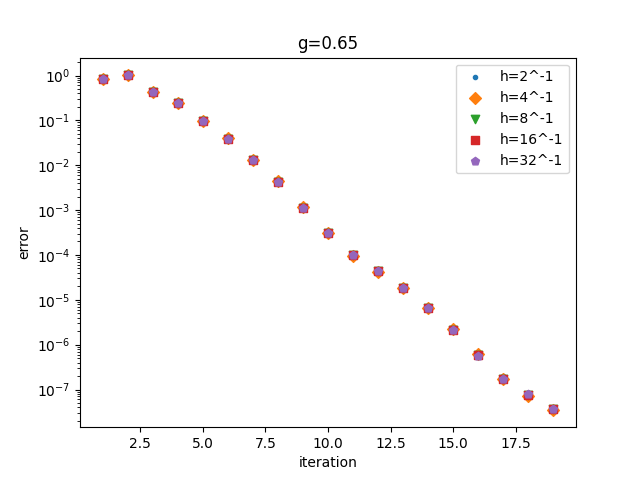}
\end{center}
\caption{$\|u_{dir}-u_n^h\|_{L^2(\Omega)}$}
\label{fig:nonlinear}
\end{figure}

For this problem, we moreover compare computing times for \eqref{eq:numScheme} and Algorithm \eqref{alg:alternative} in Table \ref{tab:timing}. For \eqref{eq:numScheme} we use $\|u_{n+1}-u_n\|_{H^1(\Omega)}\leq 10^{-9}$ as our termination condition, while for Algorithm \ref{alg:alternative} we choose $\tp{tol}_k = 10^{-(i+4)}$ when $h=2^{-i}$.
\begin{table}[ht]
\begin{center}
\begin{tabular}{ c|c|c|c|c|c }
$h$ & $2^{-1}$ & $2^{-2}$ & $2^{-3}$ & $2^{-4}$ & $2^{-5}$ \\
\hline
iteration scheme & $0.201$ & $0.715$ & $3.931$ & $47.48$ & $821.4$\\
\hline
Algorithm \ref{alg:alternative} & $0.333$ & $1.000$ & $3.3365$ & $33.41$ & $576.9$\\
\end{tabular}
\end{center}
\caption{runtime in seconds}\label{tab:timing}
\end{table}
\begin{remark}
We do not apply a scheme optimised for the $p$-Laplacian in order to solve each iteration step. Improving our computation in this way a further reduction in runtime should be achieved.
\end{remark}

\section{Improved regularity results}\label{sec:higherReg}
In this section, we assume that $$\label{eq:pLaplaceField}b(x,z)= (B(x)z\cdot z)^\frac{p-2}{2}B(x)z,$$ where $B$ is measurable and satisfies for $x\in\Omega$, $\xi\in\R^{N \times n}$,
\begin{align}\label{ass:Ax}
\lambda |\xi|^2\leq B(x)\xi\cdot \xi\leq \Lambda |\xi|^2.
\end{align}
We note that $b$ satisfies $\mathrm{\ref{ass:A1}-\ref{ass:A3}}$ with $\mu=0$. We will further assume that $B$ is symmetric almost everywhere in $\Omega$, so that \eqref{ass:A4} holds.

\subsection{A first example: Calder\'{o}n-Zygmund type estimates}\label{subsec:HigherDiff}
We will use fields satisfying the assumptions of Theorem \ref{thm:Kinnunen} as reference fields and apply the iterative process to obtain the following result.

\begin{theorem}\label{thm:ZygmundEstimates}
Let $1<p<\infty$. Suppose $f\in W^{1,q/(p-1)}(\Omega)$ and let $g=0$. Assume $a$ satisfies $\mathrm{\ref{ass:A1}-\ref{ass:A3}}$ and \eqref{ass:A4} with $\mu=0$. Further suppose there is $B$ such that \eqref{ass:Ax} and the assumptions of Theorem \ref{thm:Kinnunen} are satisfied. Finally assume that with this choice of $b$ the assumptions of Theorem \ref{thm:KoshelevBasic} are satisfied. If
$$C_1 =\dfrac{C_0^{(p-1)/q}\Lambda_b K_{a,b}}{(p-1)}<1,$$ where $K_{a,b}$ is given by \eqref{def:K}, then the solution $u$ of \eqref{eq:KoshProblem} satisfies
\begin{align*}
\|\D u\|_{L^q(\Omega)}\lesssim 1+\|f\|_{L^{q/(p-1)}}^\frac{1}{p-1}.
\end{align*}
Here $C_0$ is the constant arising in Theorem \ref{thm:Kinnunen}.
\end{theorem}
\begin{proof}

Let $\{u_n\}$ be the sequence generated by the iterative process with $u_0 = 0$ and with the choice of $\gamma$ as in in Theorem \ref{thm:KoshelevBasic}. By Theorem \ref{thm:Kinnunen}, Lemma \ref{lem:normEstimate} and \ref{ass:A3}, we have, for some $c>0$, the estimate
\begin{align*}
\int_\Omega |\D u_{n+1}|^q\d x \leq& C_0 \int_\Omega |b(x,\D u_n)-\gamma a(x,\D u_n)+\gamma f|^\frac{q}{p-1}\d x\\
\leq& C_0 \int_\Omega |b(x,\D u_n)-b(x,0)-\gamma (a(x,\D u_n)+a(x,0))+\gamma f|^\frac{q}{p-1}\d x+ c\\
\leq& C_0\Lambda_b^\frac{q}{p-1} K^\frac{q}{p-1}\int_\Omega \left(\int_0^1 \theta^{p-2}\d\theta\right)^\frac{q}{p-1} |\D u_n|^q + C_0\int_\Omega |f|^\frac{q}{p-1}\d x+c\\
=& \frac{C_0\Lambda_b^\frac{q}{p-1} K^\frac{q}{p-1}}{(p-1)^{q/(p-1)}} \int_\Omega |\D u_n|^q+C_0\int_\Omega |f|^\frac{q}{p-1}\d x+c.
\end{align*}
Thus, as by assumption $C_1<1$, we find by induction
\begin{align*}
\|\D u_{n+1}\|_{L^q(\Omega)}\lesssim \frac{C_1^\frac{q}{p-1}}{1-C_1^\frac{q}{p-1}} (1+\|f\|_{L^{q/(p-1)}}^\frac{1}{p-1}).
\end{align*}
Extracting a weakly convergent subsequence and noting that $u_n\to u$ in $W^{1,p}(\Omega)$ by Theorem \ref{thm:KoshelevBasic}, where $u$ is the solution of \eqref{eq:KoshProblem}, we find the desired estimate holds.
\end{proof}
\subsection{A second example: Weighted estimates and H\"older continuity}
We can use the Koshelev iteration to perturb Theorem \ref{thm:higherRegMengesha} as follows:

\begin{theorem}\label{thm:KoshelevWeighted}
Suppose $a$ is a field satisfying $\mathrm{\ref{ass:A1}-\ref{ass:A3}}$ and \eqref{ass:A4} with $\mu=0$ and moreover there is symmetric $B$ satisfying the assumptions of Theorem \ref{thm:higherRegMengesha} such that the assumptions of Theorem \ref{thm:KoshelevBasic} apply with this choice of $b(x,z)$. If
\begin{align*}
C_3 =\dfrac{C_2^\frac{p-1}{q} K_{a,b} \Lambda_b}{(p-1)}<1,
\end{align*}
where $K_{a,b}$ is given by \eqref{def:K}, then the solution $v$ of the boundary value problem \eqref{eq:KoshProblem} satisfies the estimate
\begin{align*}
\|\D v\|_{L^q_w(\Omega)}\lesssim 1+\|f\|_{L^\frac{q}{p-1}_w(\Omega)}^\frac{1}{p-1}.
\end{align*}
Here $C_2$ is the constant arising in Theorem \ref{thm:higherRegMengesha}.
\end{theorem}
\begin{proof}
The proof is similar to the proof of Theorem \ref{thm:ZygmundEstimates}. Let $\{u_n\}$ be given by the iterative process with $u_0 = 0$ and $\gamma$ chosen as in Theorem \ref{thm:KoshelevBasic}.

Note that $b(x,0)-\gamma a(x,0)+\gamma f\in L^\frac{q}{p-1}_w(\Omega)$. Thus we find,
\begin{align*}
\|\D u_{n+1}\|_{L^q_w(\Omega)}^q \leq& C_2 \int_\Omega |b(x,\D u_n)-\gamma a(x,\D u_n)+\gamma f|^\frac{q}{p-1}w(x)\d x\\
\leq& C_2 \int_\Omega |b(x,\D u_n)-b(x,0)-\gamma (a(x,\D u_n)-a(x,0))|^\frac{q}{p-1}w(x)\d x \\
&+C_2 |b(x,0)-\gamma a(x,0)+\gamma f|^\frac{q}{p-1}w(x)\d x\\
\leq& C_2 K_{a,b}^\frac{q}{p-1}\Lambda_b^\frac q {p-1}2^{-\frac q {p-1}} \|\D u_n\|_{L^q_w(\Omega)}^q+ c(C_2,p,q)\left(1+\|f\|_{L^\frac{q}{p-1}_w(\Omega)}^\frac{q}{p-1}\right).
\end{align*}
Hence by induction, as by assumption $C_3<1$,
\begin{align*}
\|\D u_{n+1}\|_{L^q_w(\Omega)}\lesssim \dfrac{C_3^\frac{q}{p-1}}{1-C_3^\frac{q}{p-1}}\left(1+\|f\|_{L^\frac{q}{p-1}_w(\Omega)}^\frac q {p-1}\right).
\end{align*}
Extracting a weakly convergent subsequence and noting that $u_n\to u$ in $W^{1,p}(\Omega)$ by Theorem \ref{thm:KoshelevBasic} where $u$ solves \eqref{eq:KoshProblem} we conclude the desired estimate, after passing to the limit in the estimate.
\end{proof}

We note that of particular interest is the choice $w(x) = |x|^\alpha$ which can be used to obtain estimates in the the familiar Morrey spaces and hence through the Morrey-Sobolev embedding allows to obtain continuity statements. Recall the definition of the $L^{q,\theta}$-Morrey-norm:
\begin{align*}
\|u\|_{L^{q,\theta}(\Omega)} = \sup_{0<r<\text{diam}(\Omega),z\in\Omega}r^\frac{\theta-n}{q}\|u\|_{L^q(B_r(z)\cap\Omega)},
\end{align*}
where $\theta\in(0,n)$.
We assume that all the assumptions and the notation of Theorem \ref{thm:KoshelevWeighted} hold and show how to deduce estimates in Morrey spaces. 

Fix $z\in\Omega$, $r\in(0,\text{diam}(\Omega))$ and choose for $\rho\in(0,\theta)$,
\begin{align*}
w(x)= \min\left(|x-z|^{-n+\theta-\rho},r^{-n+\theta-\rho}\right).
\end{align*}
Then by our previous work,
\begin{align*}
\|\D u\|_{L^{q,\theta}(B_r(z)\cap\Omega)}^q\leq r^{n-\theta+\rho}\|\D u\|_{L^q_w(B_r(z)\cap\Omega)}^q\leq r^{n-\theta+\rho}c(C,p,q)\left(1+\|f\|_{L^\frac{q}{p-1}_w(\Omega)}^\frac q {p-1}\right).
\end{align*}
It remains to estimate $\|f\|_{L^\frac{q}{p-1}_w(\Omega)}$. For this we proceed exactly as \cite{Mengesha2012} but provide the argument here for the sake of completeness. We will show that
$\|f\|_{L^\frac{q}{p-1}_w(\Omega)}\leq c \|f\|_{L^{q,\theta}(\Omega)}^q r^{-\rho}$, which will conclude the proof.

It is convenient to introduce $f'$ where $|f'|^{p-2}f' = f$.
For $\alpha>0$ we denote the set ${E_\alpha = \{x\in\Omega\colon |f'|>\alpha\}}$. Then we can write
\begin{align*}
\|f\|_{L^\frac{q}{p-1}(\Omega)}^q = \|f'\|_{L^q(\Omega)}^q =& q \int_0^\infty \alpha^q \int_{E_\alpha} w(x)\d x\frac{d\alpha}{\alpha}\\
\leq& q\int_0^\infty \alpha \int_0^{r^{-n+\theta-\rho}} \left|E_\alpha \cap B_{\beta^\frac{1}{-n+\theta-\rho}}(z)\right|\d \beta \frac{d\alpha}{\alpha}.
\end{align*}
We now estimate the inner integral as follows:
\begin{align*}
\int_0^{r^{n+\theta-\rho}} \left|E_\alpha \cap B_{\beta^\frac{1}{-n+\theta-\rho
}}(z)\right|d\beta\leq& \sum_{i=1}^\infty 2^{-i}r^{-n+\theta-\beta}\left|E_\alpha \cap B_{r 2^\frac{-i}{-n+\theta-\rho}}(z)\right|\\
\leq& 2 \int_0^{\frac 1 2 r^{-n+\theta-\rho}} \beta \left|E_\alpha \cap B_{\beta^\frac{1}{-n+\beta-\rho}}(z)\right|\frac{d\beta}{\beta}.
\end{align*}
Now returning to the original estimate and applying Fubini's theorem we conclude
\begin{align*}
\|f\|_{L^\frac{q}{p-1}_w(\Omega)}^q \leq& 2q \int_0^\infty \alpha^q \int_0^{\frac{1}{2} r^{-n+\theta-\rho}}\beta \left|E_\alpha \cap B_{\beta^\frac{1}{-n+\theta-\rho}}(z)\right|\frac{\d\beta}{\beta}\frac{\d\alpha}{\alpha}\\
\leq& 2q \|f\|_{L^{q,\theta}(\Omega)}^q \int_0^{\frac 1 2 r^{-n+\theta-\rho}} \beta^{1+\frac{n-\theta}{-n+\theta-\rho}}\leq c \|f\|_{L^{q,\theta}(\Omega)}^q r^{-\rho}.
\end{align*}
This gives the desired result.

\appendix
\section{Proof of Lemma \ref{lem:VFuncEquiv}}
We restate and prove Lemma \ref{lem:VFuncEquiv} here. 

\begin{lemma}
 For $\gamma\geq 0$ and $\mu\geq 0$ we have
  \[
  \frac{1}{6^{\gamma}(2\gamma+1)}(\mu^2+|\eta|^2+|\eta-\xi|^2)^\gamma\leq \int_0^1 (\mu^2+|t\xi+(1-t)\eta|^2)^\gamma\d t\leq 2^\gamma(\mu^2+|\eta|^2+|\eta-\xi|^2)^\gamma
 \]
 If $\gamma\in (-1/2,0]$ we have
 \begin{align*}
   2^\gamma(\mu^2+|\eta|^2+|\eta-\xi|^2)^\gamma\leq \int_0^1 (\mu^2+|t\xi+(1-t)\eta|^2)^\gamma\d t\leq \frac{1}{4^\gamma(\gamma+1)}(\mu^2+|\eta|^2+|\xi-\eta|^2)^\gamma
 \end{align*}
\end{lemma}

\begin{proof}
We first consider $\gamma\geq 0$. Then by Young's inequality,
\begin{align*}
\int_0^1 (\mu^2+|t\xi+(1-t)\eta|^2)^\gamma\d t \leq& 2^{\gamma} \int_0^1 (\mu^2+|\eta|^2+t^2|\xi-\eta|^2)^\gamma\d t\\
\leq& 2^\gamma(\mu^2+|\eta|^2+|\xi-\eta|^2)^\gamma
\end{align*}

For the lower bound, note that by symmetry we may assume $|\eta|\geq |\xi|$. Dividing both sides by $|\eta|^\gamma$ we see that moreover we may assume $|\eta|=1$. Rotating coordinate axis if necessary we may even assume that $\eta =e_1$, the unit vector in the first coordinate direction. Note that
\begin{align*}
|\eta + t(\xi-\eta)|\geq |1- t|\xi-\eta||.
\end{align*}
Write $s = |\xi-\eta|$. Note that $s\in[0,2]$. 

We first assume $s\leq 1$. Then as $\gamma\geq 0$,
\begin{align*}
&\int_0^1 (\mu^2+|t\xi+(1-t)\eta|^\gamma\d t\geq \int_0^1 (\mu^2+(1-ts)^2)^\gamma \d t
\geq 2^{-\gamma} \int_0^1 (\mu+1-ts)^{2\gamma}\d t\\
=& \frac{1}{2^\gamma(2\gamma+1)s} \left((\mu+1)^{2\gamma+1}-(\mu+1-s)^{2\gamma+1}\right)\\
=& \frac{1}{2^\gamma(2\gamma+1)s}(\mu+1)^{2\gamma}\left((\mu+1)\left(1-\left(\frac{\mu+1-s}{\mu+1}\right)^{2\gamma}\right)+s\left(\frac{\mu+1-s}{\mu+1}\right)^{2\gamma}\right)\\
\geq& \frac{1}{2^\gamma(2\gamma+1)}(\mu+1)^{2\gamma}\left((\mu+1)\left(1-\left(\frac{\mu+1-s}{\mu+1}\right)^{2\gamma}\right)+\left(\frac{\mu+1-s}{\mu+1}\right)^{2\gamma}\right)\\
\geq& \frac{1}{2^\gamma(2\gamma+1)}(\mu+1)^{2\gamma}\\
\geq& \frac{1}{4^\gamma(2\gamma+1)} (\mu^2+1+s^2)^\gamma.
\end{align*}

Next assume $s\geq 1$. Then
\begin{align*}
&\int_0^1 (\mu^2+|t\xi+(1-t)\eta|^\gamma\d t\geq \int_0^1 (\mu^2+(1-ts)^2)^\gamma \d t\\
\geq& 2^{-\gamma} \int_0^{1/s}(\mu+1-ts)^{2\gamma}\d t+ 2^{-\gamma} \int_{1/s}^1 (\mu+ts-1)^{2\gamma}\\
=& \frac{1}{2^\gamma(2\gamma+1)s} \left((\mu+1)^{2\gamma+1}-\mu^{2\gamma+1}+(\mu+s-1)^{2\gamma+1}-\mu^{2\gamma+1}\right)\\
\geq& \frac{1}{2^\gamma(2\gamma+1)s}(\mu+1+s)^{2\gamma}\left(\left(\frac{\mu+1}{\mu+1+s}\right)^{2\gamma}+(s-1)\left(\frac{\mu+s-1}{\mu+1+s}\right)^{2\gamma}\right)
\end{align*}
Define
\begin{align*}
f(s)=\frac 1 s \left(\left(\frac{\mu+1}{\mu+1+s}\right)^{2\gamma}+(s-1)\left(\frac{\mu+s-1}{\mu+1+s}\right)^{2\gamma}\right)
\end{align*}
Note that
\begin{align*}
f'(s)=& \frac 1 {s^2}\left((1-\frac 2 {1+\mu+s})^{2\gamma} + 4\gamma(s-1)s \frac{(\mu+s-1)^{2\gamma-1}}{(\mu+s+1)^{2\gamma+1}}-\frac{(\mu+1)^{2\gamma}(1+\mu+s+2\gamma s)}{(\mu+s+1)^{2\gamma+1}}\right)
\\\leq 0
\end{align*}
Hence we conclude
\begin{align*}
\int_0^1 (\mu^2+|t\xi+(1-t)\eta|^\gamma\d t\geq& \frac{1}{2^\gamma(2\gamma+1)}(\mu+s)^{2\gamma}\left(\frac{\mu+1}{\mu+3}\right)^{2\gamma}\\
\geq& \frac{1}{6^\gamma(2\gamma+1)}(\mu+s)^{2\gamma}
\end{align*}
This completes the case $\gamma\geq 0$. 

We now turn to $\gamma\in (-\frac 1 2,0)$. We find
\begin{align*}
\int_0^1 (\mu^2+|t\eta+(1-t)\xi|^2)^\gamma \d t \geq& \int_0^1 (\mu^2+2(|\eta|^2+|\xi-\eta|^2)^2)^\gamma\d t\\
\geq& 2^\gamma (\mu^2+|\eta|^2+|\eta-\xi|^2)^\gamma.
\end{align*}
We turn to the upper bound. Again set $s = |\xi-\eta|$. First assume $s\leq 1$.Then
\begin{align*}
\int_0^1(\mu^2+|t\xi+(1-t)\eta|)^\gamma\d t \leq& \int_0^1 (\mu^2+(1-ts)^2)^\gamma \d t\\
\leq& 2^{-\gamma} \int_0^1 (\mu+1-ts)^{2\gamma}\\
=& \frac{1}{2^\gamma(2\gamma+1)s}\left((\mu+1)^{2\gamma+1}-(\mu+1-s)^{2\gamma+1}\right).
\end{align*}
Set
\begin{align*}
f(s) = \frac 1 s\left((\mu+1)^{2\gamma+1}-(\mu+1-s)^{2\gamma+1}\right)
\end{align*}
and note
\begin{align*}
f'(s) = (1+\mu)((1+\mu-s)^{2\gamma}-(1+\mu)^{2\gamma})+2\gamma s(1+\mu-s)^{2\gamma}\geq 0
\end{align*}
to conclude
\begin{align*}
\int_0^1 (\mu^2+|t\eta+(1-t)\xi|^2)^\gamma \d t \leq&\frac{1}{2^\gamma(2\gamma+1)}((\mu+1)^{2\gamma+1}-\mu^{2\gamma+1})\\
\leq& \frac{(\mu+1)^{2\gamma}}{2^\gamma(2\gamma+1)}\\
\leq& \frac{(\mu+1+s^2)^{2\gamma}}{4^\gamma(2\gamma+1)}.
\end{align*}

If $s\geq 1$ we estimate
\begin{align*}
\int_0^1(\mu^2+|t\xi+(1-t)\eta|)^\gamma\d t \leq& \int_0^1 (\mu^2+(1-ts)^2)^\gamma \d t\\
\leq& 2^{-\gamma} \int_0^{1/s} (\mu+1-ts)^{2\gamma}+2^{-\gamma}\int_{1/s}^1 (\mu+ts-1)^{2\gamma}\\
=& \frac{1}{2^\gamma(2\gamma+1)s}\left(\mu+1)^{2\gamma+1}-\mu^{2\gamma+1}+(\mu+s-1)^{2\gamma+1}-\mu^{2\gamma+1}\right)\\
\leq& \frac{1}{2^\gamma(2\gamma+1)s}\left((\mu+1)^{2\gamma}+(s-1)(\mu+s-1)^{2\gamma}\right)
\end{align*}
Define
\begin{align*}
g(s)=\frac 1 s \left((\frac{\mu+1}{\mu+s})^{2\gamma}+(s-1)\left(\frac{\mu+s-1}{\mu+s}\right)^{2\gamma}\right)
\end{align*}
and note that
\begin{align*}
&0\leq g'(s) = \frac 1 {s^2 (\mu+ s-1)}(m + s)^{-1 - 
  2 \gamma} (-(\mu+1)^{2 \gamma} (\mu + s-1) (\mu + s + 2 \gamma s) \\
  &\qquad+( m + s-1)^{2\gamma} (\mu^2 + (1 + 2 \gamma) (s-1) s + \mu ( 2 s-1)))\\
  \Leftrightarrow& (\frac{\mu+1}{\mu+s-1})^{2\gamma}>\frac{\mu^2+(1+2\gamma)s(s-1)+\mu(2s-1)}{(\mu+s-1)(\mu+s+2\gamma s)}
\end{align*}
By Bernoulli's inequality it suffices to check that
\begin{align*}
&\mu+s-1+2\gamma(s-2)\geq \frac{\mu^2+(1+2\gamma)s(s-1)+\mu(2s-1)}{\mu+s+2\gamma s} \\
\Leftrightarrow& 2\gamma \mu(s-2)+2\gamma s (s-2)+4\gamma^2 s(s-2) \geq 0.
\end{align*}
It is straightforward to see that the last inequality holds.
Thus $g(s)$ is increasing and we can conclude
\begin{align*}
\int_0^1(\mu^2+|t\xi+(1-t)\eta|)^\gamma\d t \leq& \frac{(\mu+s)^{2\gamma}}{2^\gamma(2\gamma+1)}\\
\leq& \frac{(\mu+s+1)^{2\gamma}}{4^\gamma(2\gamma+1)}. 
\end{align*}
\end{proof}

\textbf{Acknowledgments:} The author would like to thank Jan Kristensen for suggesting the topic of this paper to him. He would also like to thank Ioannis Papadopoulos and Pascal Heid for useful discussion and practical suggestions regarding the numerical experiments.

{\small
\bibliographystyle{acm}
\bibliography{../bibtex/Koshelev}
}
\end{document}